\documentclass[11pt]{article} 

\usepackage[utf8]{inputenc} 

\usepackage{geometry} 
\usepackage{graphicx} 

\usepackage{booktabs} 
\usepackage{array} 
\usepackage{paralist} 
\usepackage{verbatim} 
\usepackage{subfig} 
\usepackage{amsfonts}
\usepackage{amsthm}
\usepackage{amsmath}
\usepackage{amssymb}
\usepackage{todonotes}

\usepackage{fancyhdr} 
\usepackage{sectsty}
\allsectionsfont{\sffamily\mdseries\upshape} 

\theoremstyle{definition}
\newtheorem{theorem}{Theorem}[section]
\newtheorem{corollary}[theorem]{Corollary}
\newtheorem{conjecture}[theorem]{Conjecture}

\newtheorem{lemma}[theorem]{Lemma}
\newtheorem{proposition}[theorem]{Theorem}

\newtheorem{example}[theorem]{Example}
\newtheorem{definition}[theorem]{Definition}
\newtheorem{remark}[theorem]{Remark}
\newtheorem{fact}[theorem]{Fact}

\newcommand{\Mod}{\textrm{Mod}}
\newcommand{\iso}{\cong}

\newcommand{\ptl}{\textrm{ptl}}

\newcommand{\sat}{\textrm{sat}}

\newcommand{\HC}{\textrm{HC}}
\newcommand{\CSS}{\textrm{CSS}}
\newcommand{\css}{\textrm{css}}

\makeatletter
\def\blfootnote{\xdef\@thefnmark{}\@footnotetext}
\makeatother

\begin{document}
\author{Danielle Ulrich \thanks{The author was partially supported by NSF Research Grants DMS-1308546 and DMS-2154101.}
	 \\ Department of Mathematics\\University of Maryland, College Park}
	\title{Borel Complexity and the Schr\"{o}der--Bernstein Property}
	\date{\today} 
	\maketitle

	\bibliographystyle{plain}

	\begin{abstract}
	We introduce a new invariant of Borel reducibility, namely the notion of thickness; this associates to every sentence $\Phi$ of $\mathcal{L}_{\omega_1 \omega}$ and to every cardinal $\lambda$, the thickness $\tau(\Phi, \lambda)$ of $\Phi$ at $\lambda$. As applications, we show that all the Friedman--Stanley jumps of torsion abelian groups are not Borel complete. We also show that under the existence of large cardinals, if $\Phi$ is a sentence of $\mathcal{L}_{\omega_1 \omega}$ with the Schr\"{o}der--Bernstein property (that is, whenever two countable models of $\Phi$ are bi-embeddable, then they are isomorphic), then $\Phi$ is not Borel complete.
	\end{abstract}
\section{Introduction}

In their seminal paper \cite{FS}, Friedman and Stanley introduced \emph{Borel complexity}, a measure of the complexity of the class of countable models of a sentence $\Phi$. Let $\mbox{Mod}(\Phi)$ be the set of all countable models of $\Phi$ with universe $\mathbb{N}$ (or any other fixed countable set). Then $\mbox{Mod}(\Phi)$ can be made into a standard Borel space in a natural way. 

\begin{definition}
Suppose $\Phi$, $\Psi$ are sentences of $\mathcal{L}_{\omega_1 \omega}$. Then say that $\Phi \leq_B \Psi$ ($\Phi$ is {\em Borel reducible} to $\Psi$) if there is a Borel-measurable function $f: \mbox{Mod}(\Phi) \to \mbox{Mod}(\Psi)$ satisfying the following: for all $M_1, M_2 \in \mbox{Mod}(\Phi)$, $M_1 \cong M_2$ if and only if $f(M_1) \cong f(M_2)$.
\end{definition}

 One way to think about this is that $f$ induces an injection from $\mbox{Mod}(\Phi)/\cong$ to $\mbox{Mod}(\Psi)/\cong$; in other words, we are comparing the \emph{Borel cardinality} of $\mbox{Mod}(\Phi)/\cong$ and $\mbox{Mod}(\Psi)/\cong$.

In \cite{FS}, Friedman and Stanley showed that there is a maximal class of sentences under $\leq_B$, namely the \emph{Borel complete sentences}. For example, the theories of groups, rings, linear orders, and trees are all Borel complete. Note that if $\Phi$ is Borel complete, then classifying the countable models of $\Phi$ is as hard as classifying arbitrary countable structures, so it is reasonable to say that there is no satisfactory classification. Using methods of descriptive set theory, Friedman and Stanley additionally proved that several sentences are not Borel complete. For example, they introduced the important Friedman--Stanley tower and showed that it is strictly increasing under Borel complexity.

In joint work with Richard Rast and Chris Laskowski \cite{URL}, we developed and applied some new machinery (closely related to the theory of pinned names \cite{Kanovei}) to prove non-reducibility results. Namely, to every sentence of $\Phi$ of $\mathcal{L}_{\omega_1 \omega}$, we associate a certain cardinal invariant $\|\Phi\|$, called its potential cardinality.



The following theorem is the foundation of our results in \cite{URL}:

\begin{theorem}
If $\Phi \leq_B \Psi$, then $\|\Phi\| \leq \|\Psi\|$.
\end{theorem}

Thus, to show $\Phi \not \leq_B \Psi$ it suffices to show $\|\Phi\| > \|\Psi\|$; this is a concrete counting problem. In \cite{URL} we give several applications, including the first example of a non-Borel complete first order theory with non-Borel isomorphism relation, and also a new proof that the Friedman--Stanley tower is strictly increasing under $\leq_B$. In more detail: for any sentence $\Phi$ of $\mathcal{L}_{\omega_1 \omega}$ and for any $\alpha < \omega_1$, we define the tower of Friedman--Stanley jumps $(J^\alpha(\Phi): \alpha < \omega_1)$. The special case where $\Phi = \mbox{Th}(\mathbb{Z}, S)$ (or any other sentence of $\mathcal{L}_{\omega_1 \omega}$ with $\aleph_0$-many countable models) gives the Friedman--Stanley tower $(\Phi_\alpha: \alpha < \omega_1)$. We can compute each $\|\Phi_\alpha\| = \beth_\alpha$. This gives a conceptually clear proof that for all $\alpha < \beta$, $\Phi_\alpha <_B \Phi_\beta$; Friedman and Stanley originally proved this using a Borel determinacy argument \cite{FS}.

In contexts where the isomorphism relation of $\Phi$ is not Borel, it is often desirable to use a coarser reducibility notion than $\leq_B$. In \cite{URL}, we introduce the ordering $\leq_{\HC}$; which is actually too coarse for our present purposes. Another choice is $\leq_{a \Delta^1_2}$, the reducibility notion of absolute $\Delta^1_2$-reducibility; this was introduced by Hjorth \cite{Hjorth2}. Roughly, $\Phi \leq_{a \Delta^1_2} \Psi$ if there is an injection from $\mbox{Mod}(\Phi)/\cong$ to $\mbox{Mod}(\Psi)/\cong$ which is induced by a $\Delta^1_2$-function $f: \mbox{Mod}(\Phi) \to \mbox{Mod}(\Psi)$, and moreover, $f$ continues to work in every forcing extension.  In Section~\ref{BCSBPlethoraSec}, we introduce the slight coarsening $\leq_{a \Delta_1}$ of $\leq_{a \Delta^1_2}$. This is most convenient for our purposes.
%

One limitation of potential cardinality is that there exist sentences $\Phi$ for which $\|\Phi\| = \infty$ and yet $\Phi$ is not Borel complete. For example, let $\mbox{TAG}_1 \in \mathcal{L}_{\omega_1 \omega}$ describe torsion abelian groups. Friedman and Stanley showed in \cite{FS} that ${\mbox{TAG}}_1$ is not Borel complete, and in fact that $\Phi_2 \not \leq_B {\mbox{TAG}}_1$. For each $1 \leq \alpha < \omega_1$, let  ${\mbox{TAG}}_{1+\alpha} = J^\alpha({\mbox{TAG}}_1)$ for each $\alpha$. We wish to generalize Friedman and Stanley's theorem \cite{FS} that $\Phi_2 \not \leq_B \mbox{TAG}_1$ to show that for all $1 \leq \alpha < \omega_1$, $\Phi_{\alpha+1} \not \leq_B \mbox{TAG}_\alpha$.

In Section~\ref{BCSBThicknessSec}, we introduce the notion of thickness to separate them. Namely, for each sentence $\Phi \in \mathcal{L}_{\omega_1 \omega}$, we define the thickness spectrum $\tau(\Phi, \lambda)$ of $\Phi$, a function from cardinals to cardinals. The precise definition of thickness is rigged so that if $\Phi \leq_{a \Delta_1} \Psi$, then for every cardinal $\lambda$, $\tau(\Phi, \lambda) \leq \tau(\Psi, \lambda)$.

As a first application of the definition of thickness, we show the following in Section~\ref{ThickCompSection}, using technical lemmas from Sections~\ref{Technical1}, ~\ref{DensityLemmas} and ~\ref{TechnicalThree}.

\begin{itemize}
	\item[(I)] For every $\alpha < \omega_1$ and for every cardinal $\lambda$, $\tau(\Phi_\alpha, \lambda) = \beth_\alpha$;
	\item[(II)] For every $1 \leq \alpha < \omega_1$ and for every regular strong limit $\lambda$, $\tau({\mbox{TAG}}_\alpha, \lambda) = \beth_\alpha(\lambda)$;
	\item[(III)] For every Borel complete $\Phi$ and for every regular strong limit $\lambda$, $\tau(\Phi, \lambda) = \beth_{\lambda^+}$.
\end{itemize}

Note that a regular strong limit cardinal is either $\aleph_0$ or inaccessible. In particular, $\tau(\Phi_\alpha, \aleph_0) = \tau({\mbox{TAG}}_\alpha, \aleph_0) = \beth_\alpha$; thus, we obtain

\begin{corollary}For all $\alpha < \omega_1$, $\Phi_{\alpha+1} \not \leq_{a \Delta_1} {\mbox{TAG}}_\alpha$, and hence the corresponding statement holds for $\leq_B$ also, as desired.
\end{corollary}

We formulate the following conjecture.

\vspace{1 mm}

\begin{conjecture}\label{MainGapConj}Suppose $\Phi \in \mathcal{L}_{\omega_1 \omega}$. Then $\Phi$ is Borel complete if and only if $\tau(\Phi, \aleph_0) = \beth_{\omega_1}$ (the maximal possible value).

\end{conjecture}

We remark that the positive solution to this conjecture would resolve several open problems. For instance, it would imply that for every sentence $\Phi$, $\Phi$ is Borel complete if and only if its jump is (only the forward direction is known), and if $\Psi$ is obtained from $\Phi$ by adding finitely many unconstrained constant symbols, then $\Phi$ is Borel complete if and only if $\Psi$ is (neither direction is known).

We present another application of the thickness machinery, namely to the Schr\"{o}der--Bernstein property. This has been investigated for uncountable models in several works, including \cite{Nurm1} and \cite{Goodrick2}, but we are just interested in the situation for countable models.

\begin{definition} Suppose $\Phi$ is a sentence of $\mathcal{L}_{\omega_1 \omega}$. Then say that $\Phi$ has the {\em Schr\"{o}der--Bernstein property} if for all countable $M, N \models \Phi$, if $M$ and $N$ are bi-embeddable then $M \cong N$ (there is some freedom with what notion of embeddability one uses).
\end{definition}

Some initial properties of the Schr\"{o}der--Bernstein property are developed in Section~\ref{BCSBSBSec}.

In Section~\ref{BCSBCountingSec}, we prove the following. $\kappa(\omega)$, the $\omega$\emph{'th Erd{\H o}s cardinal}, is the least cardinal satisfying $\kappa \rightarrow (\omega)^{<\omega}_2$; $\kappa(\omega)$ cannot be proven to exist in $ZFC$, but it is relatively low in the hierarchy of large cardinal axioms.

\begin{theorem}\label{SBThickBound1}
	Assume $\kappa(\omega)$ exists, and suppose $\Phi$ has the Schr\"{o}der--Bernstein property. Then $\tau(\Phi, \kappa(\omega)) \leq \kappa(\omega)$, so in particular ${\mbox{TAG}}_1 \not \leq_B \Phi$. 
\end{theorem}

In Section \ref{ExamplesSec}, we use this machinery and also a result from \cite{TFAGWShelah} to show the following. Let Graphs denote the theory of graphs, i.e. an sets with an irreflexive symmetric relation. Let $\mbox{CT}_\omega$ denote the class of $\omega$-colored trees, i.e. trees (of height at most $\omega$) with a coloring by $\omega$ (formally given by disjoint unary predicates) 
\begin{corollary}
  Suppose $ZFC + ``\kappa(\omega)$ exists" is consistent. Then it is independent of $ZFC$ whether there is an absolutely $\Delta^1_2$-reduction $F$ from Graphs to $\mbox{CT}_\omega$, such that $F$ takes non-isomorphic graphs to non-bi-embeddable trees.
\end{corollary}

\section{Preliminaries and notation}
In terms of variables, $\phi, \psi, \Phi, \Psi$ denote formulas, $P, Q$ denote forcing notions, $\lambda, \kappa, \mu$ denote cardinals, $n, m$ denote natural numbers, $\alpha, \beta, \gamma$ denote ordinals, $\delta$ denotes a limit ordinal, $M, N$ denote structures, $V$ denotes a transitive model of set theory, $\mathbb{V}$ denotes the universe of all sets.

Our metatheory will always be $ZFC$ (in particular, everything is a set). Frequently we will need to work in transitive models of set theory; but there are not guaranteed to be set models of $ZFC$. $ZFC^-$ is a convenient fragment of $ZFC$ for this purpose; in fact, $ZFC^-$ has several desirable properties, which we describe now.

\begin{definition}
	Let $ZFC^-$ be $ZFC$ but: remove power set, and strengthen choice to the well-ordering principle, and strengthen replacement to the collection principle (this is as in \cite{ZFCminus}). 
	
\end{definition}

\begin{example}
	If $\lambda$ is an uncountable regular cardinal, then $H(\lambda) \models ZFC^-$, where $H(\lambda)$ is the set of sets of hereditary cardinality less than $\lambda$. Thus, if $A$ is any transitive set, then there is some transitive $V \models ZFC^-$ with $|V| = |A|+\aleph_0$. 
\end{example}

We usually denote $H(\aleph_1)$ as $\HC$.

Most arguments that do not appeal explicitly appeal to powerset go through in $ZFC^-$. For instance, successor cardinals are regular. Transfinite induction works fine. Every set $X$ is in bijection with an ordinal $\alpha$; thus, it makes sense to define the cardinality of $X$ to be the least such ordinal $\alpha$.

The following lemma must be reproven for every fragment of $ZFC$ one works with. For $ZFC^-$ it is standard, although we were not able to find an exact reference.

\begin{lemma}\label{ForcingLemma1}
	Suppose $\mathbb{V} \models ZFC^-$, and suppose $P$ is a forcing notion. Then the forcing theorem holds for $P$, in other words: we have a definable forcing relation $\Vdash_P$ in $\mathbb{V}$, and if $G$ is $P$-generic over $\mathbb{V}$, then $\mathbb{V}[G] \models \phi(\dot{a}_1, \ldots, \dot{a}_n)$ if and only if there is some $p \in G$ which forces $\phi(\dot{a}_1,  \ldots \dot{a}_n)$. Also, if $G$ is $P$-generic over $\mathbb{V}$, $\mathbb{V}[G] \models ZFC^-$. 
\end{lemma}
\begin{proof}
First of all, note that we can define the forcing relation $\Vdash_P$ via the usual clauses (using that $P$ is a set). Theorem 1.1 of \cite{ClassForcing2} implies that if $G$ is $P$-generic over $\mathbb{V}$, then $\mathbb{V}[G] \models \phi(\dot{a}_1, \ldots, \dot{a}_n)$ if and only if there is some $p \in G$ which forces $\phi(\dot{a}_1, \ldots, \dot{a}_n)$. 

So it remains to check that $P \Vdash ZFC^-$. We show $P$ forces separation, collection and well-ordering; the other axioms are also straightforward to check.

Separation: suppose $\dot{A}$, $\dot{a}_0, \ldots, \dot{a}_{n-1}$ are $P$-names, and $\phi(x, z_0, \ldots, z_{n-1})$ is a formula of set theory. Let $\dot{B} \in \mathbb{V}$ be the $P$-name consisting of all pairs $(p, \dot{b})$, such that there is some $q \in P$ with $p \leq q$ and $(q, \dot{b}) \in \dot{A}$, and such that $p \Vdash \phi(\dot{a}, \dot{a}_0, \ldots, \dot{a}_{n-1})$. Clearly $P \Vdash \dot{B} = \{b \in \dot{A}: \phi(b, \dot{a}_0, \ldots, \dot{a}_{n-1})\}$.

Collection: suppose $\dot{A}$, $\dot{a}_0, \ldots, \dot{a}_{n-1}$ are $P$-names, and $\phi(x, y, z_0, \ldots, z_{n-1})$ is a formula of set theory. By collection in $\mathbb{V}$, we can find some $P$-name $\dot{B}$ such that for every $(p, \dot{a}) \in \dot{A}$ and for every $q \leq p$, if there is some $P$-name $\dot{b}$ such that $q \Vdash \phi(\dot{a}, \dot{b}, \dot{a}_0, \ldots, \dot{a}_{n-1})$, then for some such $\dot{b}$ we have $(q, \dot{b}) \in \dot{B}$. Clearly $P$ forces this works.

Well-ordering: suppose $\dot{A}$ is a $P$-name. Let $f: \dot{A} \to \alpha$ be an injection for some ordinal $\alpha$, and let $\dot{R}$ be the $P$-name $\{(p, (\dot{b}, \alpha)): (p, \dot{b}) \in \dot{A} \mbox{ and } \alpha = f(p, \dot{b})\}$. Clearly $P$ forces that $\dot{R} \subseteq \dot{A} \times \alpha$, and the cross-sections corresponding to distinct elements of $\dot{A}$ are disjoint. Hence $P \Vdash \dot{A}$ is well-orderable.
\end{proof}

 We will use repeatedly a pair of closely related absoluteness results. The first is Shoenfield's Absoluteness Theorem \cite{Shoenfield}, see e.g. Theorem 13.15 of \cite{Jech} (the theorem there is just stated for $ZF$, but the same proof extends to $ZFC^-$.
 
 \begin{theorem}
 	Suppose $X \subseteq \mathbb{R}$ is $\Sigma^1_2$ (or $\Pi^1_2$). Then $X$ is absolute to transitive models $V \models ZFC^-$ with $\omega_1 \subseteq V$ (and hence also to forcing extensions).
 \end{theorem}

The second is the L{\'e}vy Absoluteness Principle, which has various forms (e.g., Theorem~9.1 of \cite{barwise2} or Section~4 of \cite{LevyAbs}); we give a version more convenient to us. For a proof, see \cite{URL} (it is also standard). We need the following standard definition.

\begin{definition}
Suppose $\phi(\overline{x})$ is a formula of set theory with parameters from $\HC$. Say that $\phi(\overline{x})$ is $\Delta_0$ if it uses only bounded quantifiers $\forall x \in y, \exists x \in y$. Say that $\phi(\overline{x})$ is $\Sigma_1$ if it is of the form $\exists \overline{y} \psi(\overline{x}, \overline{y})$ for some $\Delta_0$ $\psi(\overline{x}, \overline{y})$, and say that it is $\Pi_1$ if it is of the form $\forall \overline{y} \psi(\overline{x}, \overline{y})$ for some $\Delta_0$ $\psi(\overline{x}, \overline{y})$. Since set theory has pairing functions we will only ever use the case where $\overline{x}, \overline{y}$ are single variables.
\end{definition}

 \begin{theorem}  \label{sigma1} If $\mathbb{V}[G]$ is any forcing extension, and if $\phi(x)$ is a $\Sigma_1$ formula of set theory, then for every $a\in\HC$, $\HC \models \phi(a)$
 	if and only if $\HC^{\mathbb{V}[G]} \models \phi(a)$. 
 \end{theorem}

We now shift gears and review Borel reducibility. 

First, suppose $X$ and $Y$ are Polish spaces, and $E$ and $F$ are equivalence relations on $X$ and $Y$. Then say that $(X, E)$ is Borel reducible to $(Y, F)$, and write $(X, E) \leq_B (Y, F)$ if there is a Borel-measurable map $f: X \to Y$ such that $f$ induces an injection from $X/E$ to $Y/E$. Borel-measurability means that the inverse image of an open (Borel) set is Borel; this is equivalent to the graph of $f$ being Borel.

We will be interested in a special case of this set-up. Suppose $\mathcal{L}$ be a countable langauge and let $X_\mathcal{L}$ be the set of $\mathcal{L}$-structures with universe $\omega$. Endow $X_\mathcal{L}$ with the usual logic topology (with clopen sets being solution sets of first-order formulas); then $X_\mathcal{L}$ becomes a Polish space. Moreover, if $\Phi$ is a sentence of $\mathcal{L}_{\omega_1 \omega}$ then $\Mod(\Phi)$ is a Borel subset of $X_\mathcal{L}$; hence $\Mod(\Phi)$ is a standard Borel space. 
The relation $\iso_\Phi$ is the restriction of the isomorphism relation to $\Mod(\Phi)\times\Mod(\Phi)$.  When no ambiguity arises we will just write $\iso$.

If $\mathcal{L}'$ is another countable language and $\Phi'$ is a a sentence of $\mathcal{L}'_{\omega_1 \omega}$, then put $\Phi \leq_B \Phi'$ if $(\mbox{Mod}(\Phi), \cong) \leq_B \mbox{Mod}(\Psi), \cong)$.

The definition below is  in both Barwise~\cite{barwise2}
and Marker~\cite{MarkerMT}. 


\begin{definition} \label{cssDefinition} Suppose $\mathcal{L}$ is countable and $M$ is any  infinite $\mathcal{L}$-structure, say of power $\kappa$.
	For each $\alpha<\kappa^+$, define an $\mathcal{L}_{\kappa^+,\omega}$ formula $\phi_\alpha^{\overline{a}}({\overline{x}})$ for each finite ${\overline{a}}\in M^{<\omega}$ as follows:
	\begin{itemize}
		\item  $\phi_0^{\overline{a}}({\overline{x}}):=\bigwedge \{\theta({\overline{x}}): \theta$ atomic or negated atomic and $M\models\theta({\overline{a}})\}$;
		\item  $\phi_{\alpha+1}^{\overline{a}}({\overline{x}}):=\phi_\alpha^{\overline{a}}({\overline{x}})\ \wedge\  \bigwedge \left\{ \exists y\, \phi_\alpha^{{\overline{a}}, b}({\overline{x}},y):b\in M\right\}\ \wedge\forall y\bigvee \left\{\phi_\alpha^{{\overline{a}}, b}({\overline{x}},y):b\in M\right\}$;
		\item  For $\alpha$ a non-zero limit, $\phi_\alpha^{\overline{a}}({\overline{x}}):=\bigwedge \left\{\phi_\beta^{\overline{a}}({\overline{x}}):\beta<\alpha\right\}$.
	\end{itemize}
	Next, let $\alpha^*(M)<\kappa^+$ be least  ordinal $\alpha$ such that 
	for all finite ${\overline{a}}$ from $M$, 
	\[\forall {\overline{x}} [\phi_\alpha^{\overline{a}}({\overline{x}})\rightarrow\phi_{\alpha+1}^{{\overline{a}}}({\overline{x}})].\]

	Finally, put 
	$\css(M):=\phi^{\emptyset}_{\alpha^*(M)} \wedge \bigwedge \left\{ \forall\, {\overline{x}} [\phi^{{\overline{a}}}_{\alpha^*(M)}({\overline{x}})\rightarrow \phi^{{\overline{a}}}_{\alpha^*(M)+1}({\overline{x}})]:{\overline{a}}\in M^{<\omega}\right\}$. We call this the {\em canonical Scott sentence} of $M$.
\end{definition}

We summarize the well-known, classical facts about canonical Scott sentences with the following:

\begin{fact}  \label{summ}  Fix a countable language $\mathcal{L}$.
	\begin{enumerate}
		\item  For every $\mathcal{L}$-structure $M$, $M\models \css(M)$; and for all 
		$\mathcal{L}$-structures $N$, $M\equiv_{\infty,\omega} N$ if and only if $\css(M)=\css(N)$ if and only if $N\models \css(M)$.
		\item  If $M$ is countable, then $\css(M)\in \HC$.
		\item  The map $\css$ is absolute to transitive models of $ZFC^-$. 
		\item If $M$ and $N$ are both countable, then $M\iso N$ if and only if $\css(M)=\css(N)$ if and only if $N\models \css(M)$.
	\end{enumerate}
\end{fact}

We make the following definition. Note that $\mbox{CSS}(\Phi)$ is always in natural bijection with $\mbox{Mod}(\Phi)/\cong$.

\begin{definition}  For $\Phi$ any sentence of $\mathcal{L}_{\omega_1,\omega}$, let $\CSS(\Phi) :=\{\css(M):M\in\Mod(\Phi)\}$, a subset of $\HC$.
\end{definition}

We will also be using Karp's completeness theorem: see for instance Theorem 3 of \cite{KeislerBook}.

\begin{theorem}\label{Karp}
	Suppose $\phi$ is a sentence of $\mathcal{L}_{\omega_1 \omega}$, $V \models ZFC^-$ is transitive and $\phi \in (\HC)^V$. Then $\phi$ is satisfiable if and only if $\phi$ has a model in $(\HC)^V$.
\end{theorem}

As an example, suppose $\Phi$ is a sentence of $\mathcal{L}_{\omega_1 \omega}$, and $V \models ZFC^-$ is transitive. Then $\CSS(\Phi) \cap (\HC)^V = \{\css(M): M \in \mbox{Mod}(\Phi)^V\}$, using that $(\HC)^V \models ZFC^-$.

\section{A New Reducibility Notion}\label{BCSBPlethoraSec}
 The following coarsening of $\leq_B$ was introduced by Hjorth in \cite{Hjorth2}.
\begin{definition}
	Suppose $\Phi, \Psi$ are sentences of $\mathcal{L}_{\omega_1 \omega}$. Say that $\Phi \leq_{a \Delta^1_2} \Psi$ ($a$ stands for {\em absolutely}) if there is some function $f: \mbox{Mod}(\Phi) \to \mbox{Mod}(\Psi)$ with $\Delta^1_2$ graph, such that for all $M, N \in \mbox{Mod}(\Phi)$, $M \cong N$ if and only if $f(M) \cong f(N)$, and such that further, this continues to hold in any forcing extension. Explicitly, we require that $f$ has a $\Pi^1_2$-definition $\sigma(x, y)$, and a $\Sigma^1_2$-definition $\tau(x, y)$, such that if $\mathbb{V}[G]$ is any forcing extension, then $\sigma(x, y)$ and $\tau(x, y)$ coincide on $\mbox{Mod}(\Phi)^{\mathbb{V}[G]} \times \mbox{Mod}(\Psi)^{\mathbb{V}[G]}$ and define the graph of a function $f^{\mathbb{V}[G]}$, such that for all $M, N \in \mbox{Mod}(\Phi)^{\mathbb{V}[G]}$, $M \cong N$ if and only if $f^{\mathbb{V}[G]}(M) \cong f^{\mathbb{V}[G]}(N)$. 
\end{definition}

As discussed in the last section, the class of countable models of a theory up to isomorphism is in canonical bijection with the class of canonical Scott sentences. For our purposes it is much more convenient to work directly with canonical Scott sentences, rather than the equivalence relation of isomorphism on models. We thus formulate a coarsening of $\leq_{a \Delta^1_2}$. Since we are interested in proving nonreducibility results, this is no loss of generality.

\begin{definition}

If $\phi(x)$ is a formula of set theory with parameters from $\HC$ and $V$ is a transitive set containing the parameters for $\phi$, then let $\phi(V)$ denote $\{a \in V: V \models \phi(a)\}$.

    Say that $A \subseteq \HC$ is {\em absolutely} $\Delta_1$ (or $a \Delta_1$) if there are formulas of set theory $\phi(x), \psi(x)$, possibly with parameters from $\HC$, such that $\phi(x)$ is $\Sigma_1$, $\psi(x)$ is $\Pi_1$, and $\phi(\HC) = \psi(\HC) = A$, and in every forcing extension $\mathbb{V}[G]$, $\phi(\HC^{\mathbb{V}[G]}) = \psi(\HC^{\mathbb{V}[G]})$.
\end{definition}

We first show that the choice of definitions $\phi$, $\psi$ don't matter.

\begin{lemma}
Suppose $A \subseteq \HC$ is absolutely $\Delta_1$ via $\phi(x), \psi(x)$ and also via $\phi'(x), \psi'(x)$, so $\phi(x), \phi'(x)$ are $\Sigma_1$ and $\psi(x), \psi'(x)$ are $\Pi_1$. Then in every forcing extension $\mathbb{V}[G]$, $\phi(\HC^{\mathbb{V}[G]}) = \psi(\HC^{\mathbb{V}[G]}) = \phi'(\HC^{\mathbb{V}[G]}) = \psi'(\HC^{\mathbb{V}[G]})$.
\end{lemma}
\begin{proof}
    We know $\phi(\HC)^{\mathbb{V}[G]} = \psi(\HC)^{\mathbb{V}[G]} = X$ say, and $\phi'(\HC^{\mathbb{V}[G]}) = \psi'(\HC^{\mathbb{V}[G]})= Y$ say. Since $\HC \models \forall x (\phi(x) \rightarrow \psi'(x))$, which is $\Pi_1$, we get by L{\'e}vy's absoluteness principle that the same holds in $\HC^{\mathbb{V}[G]}$, i.e. $X \subseteq Y$. A symmetric argument show $Y \subseteq X$, so $X = Y$ as desired. 
\end{proof}

We denote the common value as $A^{\mathbb{V}[G]}$. By Lemma 2.2 of \cite{URL} we get that if $A$ is absolutely $\Delta_1$ then $(A^{\mathbb{V}[G]}: \mathbb{V}[G] \mbox{ a forcing extension of } \mathbb{V}[G])$ is a ``strongly definable family."

The following is the main source of examples of absolutely $\Delta_1$ sets. Let $ZFC^-_*$ refer to $ZFC^-$ + every set is countable; so in particular, $\HC\models ZFC^-_*$ and every transitive $V \models ZFC^-_*$ is contained in $\HC$. $ZFC^-_*$ is a strengthening of $ZFC^-$, so being absolute to countable transitive models of $ZFC^-_*$ is a weakening of being absolute to transitive models of $ZFC^-$.

\begin{lemma}\label{ZFC-Absolute}
    Suppose $A \subseteq \HC$ is absolute to countable transitive models of $ZFC^-_*$; that is, there is a formula $\phi(x)$ of set theory over some parameter $b \in \HC$ such that $A = \phi(\HC)$ and also for every countable $V \models ZFC^-_*$ with $b \in V$, $A \cap V = \phi(V)$. Then $A$ is absolutely $\Delta_1$.
\end{lemma}
\begin{proof}
    We have that $a \in A$ if and only if for some or every countable transitive $V \models ZFC^-_*$ containing the parameters and $a$, we have $a \in \phi(V)$. 
\end{proof}

For instance, if $\Phi \in \mathcal{L}_{\omega_1 \omega}$ then $\CSS(\Phi)$, being absolute to transitive models of $ZFC^-_*$, is absolutely $\Delta_1$. 

\begin{definition}
Suppose $X_i: i < n$ are absolutely $\Delta_1$ subsets of $\HC$ and $\phi(U_i: i < n)$ is a formula of set theory with $n$ additional unary predicates. Then say that $\phi(X_i: i < n)$ {\em holds persistently} if in every forcing extension $\mathbb{V}[G]$, $\HC^{\mathbb{V}[G]} \models \phi((X_i)^{\mathbb{V}[G]}: i < n)$.

Suppose $f, X, Y$ are absolutely $\Delta_1$. Then say that $f: X \leq_{a \Delta_1} Y$ if persistently, $f$ is an injection from $X$ to $Y$, and write $X \leq_{a \Delta_1} Y$ to indicate the existence of such an $f$. Write $X \sim_{a \Delta_1} Y$ to indicate $X \leq_{a \Delta_1} Y \leq_{a \Delta_1} X$. Write $f: X \cong_{a \Delta_1} Y$ to indicate $f:X \to Y$ is persistently a bijection, and write $X \cong_{a \Delta_1} Y$ to indicate the existence of such an $f$.

If $\Phi$, $\Psi$ are sentences of $\mathcal{L}_{\omega_1 \omega}$ then write $\Phi \leq_{a \Delta_1} \Psi$ to indicate $\CSS(\Phi) \leq_{a \Delta_1} \CSS(\Psi)$. Similarly define $\Phi \sim_{a \Delta_1} \Psi$, $\Phi \cong_{a \Delta_1} \Psi$.
\end{definition}

It is straightforward to show that $\leq_{a \Delta_1}$ is transitive. More involved is the following.

\begin{theorem}
    If $\Phi \leq_{a \Delta^1_2} \Psi$ then $\Phi \leq_{a \Delta_1} \Psi$.
\end{theorem}
\begin{proof}
Let $f: \mbox{Mod}(\Phi) \to \mbox{Mod}(\Psi)$ be an $a \Delta^1_2$ reduction. Let $\phi(x, y), \psi(x, y)$ be the $\Sigma^1_2, \Pi^1_2$ formulas defining $f$. Let $a$ contain all the relevant parameters. Let $\Gamma$ be the fragment of set theory with parameter $a$, asserting of its model $V$ that $V \models ZFC^-_*$ and $\phi(x, y), \psi(x, y)$ define the same function from $\mbox{Mod}(\Phi)$ to $\mbox{Mod}(\Psi)$. Note that if $V \models \Gamma$ is transitive, then $V$ computes $f$ correctly, since $\Sigma^1_2$ statements are upwards absolute and $\Pi^1_2$ statements are downards absolute between transitive models of $ZFC^-$. Hence $V$ also computes the corresponding injection $f_*$ from $\CSS(\Phi)$ to $\CSS(\Psi)$ correctly. Then $(b, c) \in f_*$ if and only if for some or every countable transitive $V \models \Gamma$, with $a, b, c \in V$, we have $(b, c) \in f_*^V$. Thus $f_*$ is absolutely $\Delta_1$; we get that persistently $f_*$ is an injection from $\CSS(\Phi)$ to $\CSS(\Psi)$ because $f$ is an absolutely $\Delta^1_2$ reduction.
\end{proof}

The following observation is an adaptation of the classical Schr\"{o}der--Bernstein theorem to our context. As notation, if $f, X$ are absolutely $\Delta_1$ and $f: X \to \HC$ then we say that $f[X]$ is absolutely $\Delta_1$ to indicate that there is an absolutely $\Delta_1$-class $Y$ such that persistently, $f[X] = Y$.

\begin{theorem}\label{SBForSets}
	$X, Y \subseteq \HC$ are absolutely $\Delta_1$. Suppose $f:X \leq_{a\Delta_1} Y$ and $g: Y \leq_{a\Delta_1} X$ satisfy that $f[X], g[Y]$ are absolutely $\Delta_1$. Then $X \cong_{a\Delta_1} Y$.
\end{theorem}
\begin{proof}
First, note that $f^{-1}$ and $g^{-1}$ are absolutely $\Delta_1$, and persistently $f^{-1}: f[X] \to X$ and $g^{-1}:g[Y] \to Y$.

	Note then that whenever $X' \subseteq X$ is $\Delta_1$-absolute, so is $f[X']$, since given $y \in Y$, we have that $y \in f[X']$ if and only if $y \in f[X]$ and $f^{-1}(y) \in X'$. Similarly, whenever $Y' \subseteq Y$ is absolutely $\Delta_1$, so is $g[Y']$.
	
	So now we can apply the normal proof of the Schr\"{o}der--Bernstein theorem. We can suppose $X$ and $Y$ are disjoint. Write $Z = X \cup Y$ and write $h = f \cup g: Z \to Z$. Define $h': Z \to Z$ via: $h'(a) = h(a)$ if $h^{-n}(a)$ exists for all $n$, and otherwise, if $n$ is least such that $h^{-n}(a)$ is undefined, then define $h'(a) = h(a)$ if $n$ is odd, and define $h'(a) = h^{-1}(a)$ if $n$ is even. Clearly, $h'$ is a bijection from $Z$ to itself such that $h'[X] = Y$ and $h'[Y] = X$,  and by the above remarks, it is clear that this holds persistently.
\end{proof}

\section{Potential Cardinality and the Friedman--Stanley Tower}

In this section we pull several notions of \cite{URL} into our context, and we define the version of the Friedman Stanley tower we wish to use.

\begin{definition}
    If $X$ is absolutely $\Delta_1$, then let $X_{\ptl}$ be the class of all $a \in \mathbb{V}$ such that in some or any forcing extension $\mathbb{V}[G]$ with $a$ hereditarily countable, we have $a \in X^{\mathbb{V}[G]}$.
    
    If $X$ is absolutely $\Delta_1$, let $\|X\| = |X_{\ptl}|$; if $X_{\ptl}$ is a proper class then we write $\|X\| = |X_{\ptl}| = \infty$. Say that $X$ is {\em short} if $\|X\| < \infty$, i.e. $X_{\ptl}$ is a set. 

If $\Phi$ is a sentence of $\mathcal{L}_{\omega_1 \omega}$, then put $\|\Phi\| = \|\CSS(\Phi)\|$ and say that $\Phi$ is short if $\CSS(\Phi)$ is.
\end{definition}

Several basic properties of the ptl map are developed in \cite{URL} (under more general conditions), including Proposition 2.14 thereof:

\begin{lemma}
Suppose $X, Y,f$ are $a \Delta_1$. If persistently $f$ is a function from $X$ to $Y$, then $f_{\ptl}$ is a function from $X_{\ptl}$ to $Y_{\ptl}$. If $f$ is also persistently injective (persistently bijective), then $f_{\ptl}$ is injective (bijective).
\end{lemma}

Thus if $X \leq_{a\Delta_1} Y$, then $\|X\| \leq \|Y\|$. This provides a potent method of proving nonreductions among short sentences, which is exploited in \cite{URL}.

We discuss several examples (and set some notation).

\begin{example}
	$\omega_{\ptl} = \omega$. $(\omega_1)_{\ptl}  = $ ON (the class of all ordinals). $(\HC)_{\ptl} = \mathbb{V}$. $\mathbb{R}_{\ptl} = \mathbb{R}$.
	
	For each $\alpha < \omega_1$, let $\HC_\alpha \subseteq \HC$ be $\HC \cap \mathbb{V}_\alpha$, i.e. the set of all hereditarily countable sets of foundation rank less than $\alpha$. Then $(\HC_\alpha)_{\ptl} = \mathbb{V}_\alpha$.
	
	If $X$ is a class, then let $\mathcal{P}(X)$ denote the class of all subsets of $X$, so $\mathcal{P}(X)$ is a set precisely when $X$ is. For each ordinal $\alpha$, define $\mathcal{P}^\alpha(X)$ inductively, via $\mathcal{P}^{\alpha+1}(X) = \mathcal{P}(\mathcal{P}^\alpha(X))$, and $\mathcal{P}^{\delta}(X) = \bigcup_{\alpha < \delta} \{\alpha\} \times \mathcal{P}^\alpha(X)$ for limit $\delta$ (i.e. we are taking the disjoint union). Define $\mathcal{P}_{\kappa}(X)$ and $\mathcal{P}_{\kappa}^{\alpha}(X)$ similarly, by restricting to subsets of size less than $\kappa$.
	
	Then for any absolutely $\Delta_1$ $X \subseteq \HC$, we have that $\mathcal{P}^\alpha_{\aleph_1}(X)$ is absolutely $\Delta_1$ for every $\alpha < \omega_1$. Further, $(\mathcal{P}_{\aleph_1}^\alpha(X))_{\ptl} = \mathcal{P}^\alpha(X_{\ptl})$.
	
	We shall be particularly interested in $\CSS(\Phi)_{\ptl}$ for $\Phi \in \mathcal{L}_{\omega_1 \omega}$. We call these the potential canonical Scott sentences of $\Phi$. As explored in \cite{URL}, this always contains the class $\CSS(\Phi)_{\sat}$ of satisfiable canonical Scott sentences of $\Phi$, namely $\CSS(\Phi)_{\sat} = \{\css(M): M \models \Phi\}$, and sometimes the inclusion is strict.

\end{example}

\begin{remark}
	The reason we take disjoint unions in the definition of $\mathcal{P}^\alpha(X)$ is so that if $X \leq_{a\Delta_1} Y$, then $\mathcal{P}^\alpha_{\aleph_1}(X) \leq_{a\Delta_1} \mathcal{P}^{\alpha}_{\aleph_1}(Y)$ for all $\alpha< \omega_1$, and similarly for $\sim_{a\Delta_1}$ and $\cong_{a\Delta_1}$. This is straightforward with disjoint union, and problematic without.

\end{remark}

There are many versions of the Friedman--Stanley tower in circulation; for instance the $I_\alpha$ in \cite{FS}, the $\iso_\alpha$ in \cite{HjorthKechrisLouveau}, the $=^\alpha$ in \cite{GaoIDST}, and the $T_\alpha$ in \cite{KoerwienBRDepth}. In \cite{URL} we used the tower $(T_\alpha: \alpha < \omega_1)$ from \cite{KoerwienBRDepth}. The advantage of this is that it is a tower of complete first order theories. For the present work we prefer to use a tower $(\Phi_\alpha: \alpha < \omega_1)$ of sentences of $\mathcal{L}_{\omega_1 \omega}$. One can show that $T_n \sim_B \Phi_n$ for each $n < \omega$, and $T_\alpha \sim_B \Phi_{\alpha+1}$ for all $\alpha \geq \omega$.

The following is as defined by \cite{FS}.

\begin{definition}\label{jumpDefinition}
	Suppose $\mathcal{L}$ is a countable relational language and $\Phi\in \mathcal{L}_{\omega_1,\omega}$.  The {\em jump of $\Phi$}, written $J(\Phi)$, is a sentence of $\mathcal{L}'_{\omega_1 \omega}$ defined as follows, where $\mathcal{L}' =\mathcal{L}\cup\{E\}$ is obtained by adding a new binary relation symbol $E$ to $\mathcal{L}$. Namely $J(\Phi)$ states that $E$ is an equivalence relation with infinitely many classes, each of which is a model of $\Phi$.  If $R\in \mathcal{L}$ and $\overline{x}$ is a tuple not all from the same $E$-class, then $R(\overline{x})$ is defined to be false, so that the models are independent.
\end{definition}

There is a corresponding notion of jump that can be defined directly on equivalence relations. Given an equivalence relation $E$ on $X$, its jump is the equivalence relation $J(E)$ on $X^\omega$, defined by setting $(x_n: n \in \omega) J(E) (y_n: n  \in \omega)$ if there is some $\sigma \in S_\infty$ with $x_{\sigma(n)} E y_n$ for all $n \in \omega$. Then the previous definition of the jump can be viewed as the special case where $(X, E)$  is $(\Mod(\Phi), \cong)$.

We wish to iterate the Friedman--Stanley jump. At limit stages we must explain what we will do. In \cite{URL} we took products, but here we prefer to take disjoint unions:

\begin{definition}
	Suppose $I$ is a countable set and for each $i$, $\Phi_i$ is a sentence of $\mathcal{L}_{\omega_1,\omega}$ in the countable relational language $\mathcal{L}_i$.  The \emph{disjoint union of the $\Phi_i$}, denoted $\sqcup_i \Phi_i$, is a sentence of $\mathcal{L}_{\omega_1 \omega}$, where $\mathcal{L}=\{U_i:i\in I\}\cup \bigcup_i \mathcal{L}_i$ is the disjoint union of the $\mathcal{L}_i$'s together with new unary predicates $\{U_i: i \in I\}$. 
	
	Namely $\sqcup_i \Phi_i$ states that the $U_i$ are disjoint and exhaustive, and that exactly one $U_i$ is nonempty, and that this $U_i$ forms a model of $\Phi_i$ when viewed as an $\mathcal{L}_i$-structure.
\end{definition}
%
%
%
%

We now define the tower $(\Phi_\alpha: \alpha < \omega_1)$. Actually, we proceed more generally, starting with any base theory.
\begin{definition}
	
	Suppose $\Phi$ is a sentence of $\mathcal{L}_{\omega_1 \omega}$ and $\alpha < \omega_1$. Then we define the $\alpha$'th {\em jump}, $J^\alpha(\Phi)$, of $\Phi$ as follows. Let $J^0(\Phi) = \Phi$. Having defined $J^\alpha(\Phi)$, let $J^{\alpha+1}(\Phi) = J(J^\alpha(\Phi))$. For limit stages, let $J^\delta(\Phi) = \sqcup_{\alpha < \delta} J^\alpha(\Phi)$.
	
	Let $\Phi_\alpha = J^\alpha(\Phi_0)$ where $\Phi_0$ is some chosen sentence with countably infinitely many countable models.
\end{definition}

\begin{proposition}\label{countT}
	Suppose $\Phi \in \mathcal{L}_{\omega_1 \omega}$ and $\alpha < \omega_1$. Then: 
	
	\begin{itemize}
		\item[(A)] If $\Phi$ has infinitely many countable models, then $$\mathcal{P}^\alpha_{\aleph_1}(\mbox{CSS}(\Phi)) \leq_{a\Delta_1}\CSS(J^\alpha(\Phi)) \leq_{a\Delta_1} \mathcal{P}^\alpha_{\aleph_1}(\omega \times \mbox{CSS}(\Phi)).$$
		\item[(B)] If $\mbox{CSS}(\Phi) \sim_{a\Delta_1} \omega \times \mbox{CSS}(\Phi)$, then $\mathcal{P}^\alpha_{\aleph_1}(\mbox{CSS}(\Phi)) \sim_{a\Delta_1} \CSS(J^\alpha(\Phi))$.
		\item[(C)] If $\mbox{CSS}(\Phi) \cong_{a\Delta_1} \omega \times \mbox{CSS}(\Phi)$, then $\mathcal{P}^\alpha_{\aleph_1}(\mbox{CSS}(\Phi)) \cong_{a\Delta_1} \CSS(J^\alpha(\Phi))$.
		\item[(D)] If $\Phi$ has infinitely many models, then $\|J^\alpha(\Phi)\| = \beth_\alpha(\|\Phi\|)$. 
	\end{itemize}
\end{proposition}
\begin{proof}
	It suffices to verify (A), while noting (towards applying Theorem~\ref{SBForSets} to (C)) that the images of all the embeddings we construct are $ZFC^-_*$-absolute. 
	
	So we prove (A).
	
	To check that $\mathcal{P}^\alpha_{\aleph_1}(\mbox{CSS}(\Phi)) \leq_{a\Delta_1} \CSS(J^\alpha(\Phi))$ is a routine inductive argument. To show that $\CSS(J^\alpha(\Phi)) \leq_{a\Delta_1} \mathcal{P}^\alpha(\omega \times \CSS(\Phi))$,  we need to handle multiplicities. 
	
	First we find reductions  $f_\beta: \omega \times \mathcal{P}^\beta_{\aleph_1}(\omega \times \CSS(\Phi)) \leq_{a\Delta_1} \mathcal{P}^\beta_{\aleph_1}(\omega \times \CSS(\Phi))$, for each $\beta \leq \alpha$. First, if $a = (m, \phi) \in \omega \times \mbox{CSS}(\Phi)$, then define $f_0(n, \phi)$ to be $(\langle n, m \rangle, \phi)$. Next, having defined $f_\beta$, and given $(n, a) \in \omega \times \mathcal{P}^{\beta+1}_{\aleph_1}(\omega \times \CSS(\Phi))$, define $f_{\beta+1}(n, a) = \{f_\beta(n, b): b \in a\}$. Finally, suppose we have defined $f_\beta$ for each $\beta< \delta$ limit. Given $(n, a) \in \omega \times \mathcal{P}^\delta_{\aleph_1}(\omega \times \CSS(\Phi))$, write $a= (\beta, b)$ for some $b \in \mathcal{P}^\beta_{\aleph_1}(\omega \times \CSS(\Phi))$, and define $f_\delta(n , a) = (\beta, f_{\beta}(n , b))$. 
	
	Now we define reductions $g_\beta: \CSS(J^\beta(\Phi)) \leq_{a\Delta_1} \mathcal{P}^\beta_{\aleph_1}(\omega \times \CSS(\Phi))$ for each $\beta \leq \alpha$. For $\beta = 0$, let $g_0(\phi) = (0, \phi)$. Having defined $g_\beta$, define $g_{\beta+1}: \CSS(J^{\beta+1}(\Phi)) \leq_{a\Delta_1}\mathcal{P}^{\beta+1}_{\aleph_1}(\omega \times \CSS(\Phi))$ as follows: suppose $\phi \in \CSS(J^{\beta+1}(\Phi))$, and let $M \models \phi$; enumerate the equivalence classes of $M$ as $(M_n: n < \omega)$, so each $M_n \models J^\beta(\Phi)$. Now, let $X = \{\css(M_n): n < \omega\}$ and for each $\psi \in X$, let $n_\psi = |\{n < \omega: \css(M_n) = \psi\}|$. Define $g_{\beta+1}(\phi) = \{f_\beta(n_\phi, g_\beta(\psi)): \psi \in X\}$. The limit stage is similar. 
\end{proof}

The following corollary is proved in \cite{URL} in the more general context of $\leq_{\HC}$:

\begin{corollary}\label{countTCor}
	Suppose $\Phi \in \mathcal{L}_{\omega_1 \omega}$. Then for all $\alpha < \beta$, $J^\alpha(\Phi) \leq_{B} J^\beta(\Phi)$. If $\Phi$ is short with more than one countable model, then for all $\alpha < \beta$, $J^\alpha(\Phi)<_{a\Delta_1} J^\beta(\Phi)$ (and hence this is true for $<_B$ as well).
\end{corollary}

It is often said informally that $\Phi_2$ (or $F_2$, or $=^+$) can be identified with countable sets of reals. We can make this literally true with the following theorem.

\begin{theorem}
	For all $\alpha < \omega_1$, $\mbox{CSS}(\Phi_\alpha) \cong_{a\Delta_1} \mathcal{P}^\alpha_{\aleph_1}(\omega) \cong_{a\Delta_1} \HC_{\omega+\alpha} $.
\end{theorem}
\begin{proof}
Let $\mbox{HF}$ denote the hereditarily finite sets; then $\CSS(\Phi_0) \cong_{a\Delta_1} \omega \times \CSS(\Phi_0) \cong_{a\Delta_1} \omega$, so we conclude by Theorem \ref{countT} that for each $\alpha$, $\CSS(\Phi_\alpha) \cong_{a\Delta_1} \mathcal{P}^\alpha_{\aleph_1}(\CSS(\Phi_0)) \cong_{a\Delta_1} \mathcal{P}^\alpha_{\aleph_1}(\omega)$. But $\mathcal{P}^\alpha_{\aleph_1}(\mbox{HF}) \cong_{a\Delta_1} \HC_{\omega+\alpha}$ easily (using Theorem~\ref{SBForSets}), so we are done.
\end{proof}

We can thus identify $\mbox{CSS}(\Phi_\alpha)$ with either $\mathcal{P}^\alpha_{\aleph_1}(\omega)$ or else $\HC_{\omega+\alpha}$, whichever is convenient.
\section{Thickness}\label{BCSBThicknessSec}
In this section we motivate and then define the key technical concept of the paper.

We would love to use counting arguments to characterize Borel completeness. Potential cardinality is not enough: there are examples of relatively nice $\Phi$ that are not short, so potential cardinality says nothing about them. For instance, let $\mbox{TAG}_1$ be the sentence of $\mathcal{L}_{\omega_1 \omega}$ describing torsion abelian groups. We recall some group-theoretic facts:

	For each $p$, let $\mbox{TAG}_{1, p}$ denote the sentence of $\mathcal{L}_{\omega_1 \omega}$ describing abelian $p$-groups; that is, abelian groups $A$ such that for every $a \in A$, we have $p^n a = 0$ for some $n$. It is a standard fact that if $A$ is torsion abelian, then $A$ decomposes uniquely as the direct sum of $p$-torsion groups over all primes $p$. Hence $\mbox{TAG}_1 \cong \prod_{p} \mbox{TAG}_{1, p}$.
	
	Ulm classified torsion abelian groups up to isomorphism in \cite{Ulm}, introducing what is now called the Ulm analysis: we follow the notation of \cite{Thomas2}. Suppose $A$ is a countable abelian $p$-group. (Ulm's analysis actually works for any abelian $p$-group.) Define $(A^\alpha: \alpha \in \mbox{ON})$ inductively as follows: $A^0 = A$, $A^{\alpha+1} = \bigcap_{n < \omega} p^n A^{\alpha}$, and take intersections at limit stages. Let $\tau(A) < \omega_1$ be least so that $A^{\tau(A)} = A^{\tau(A)+1}$. For each $\alpha < \tau(A)$, let $A_\alpha = A^\alpha / A^{\alpha+1}$. Then each $A_\alpha$ is a direct sum of cyclic $p$-groups, and so can be written uniquely as $\oplus_{n \geq 1} (\mathbb{Z}/p^n\mathbb{Z})^{m^\alpha_n(A)}$ where $0 \leq m^\alpha_n(A) \leq \omega$. Also, $A^{\tau(A)}$ is divisible, and hence is determined by its rank $\mbox{rnk}(A^{\tau(A)})$, a number between $0$ and $\omega$. Finally, $A$ is determined up to isomorphism by $(\tau(A), \mbox{rk}(A^{\tau(A)}), A_\alpha/\cong: \alpha < \tau(A))$.

It is often send informally that the countable models of $\mbox{TAG}_1$ are classified by countable subsets of $\omega_1$. We can make this a literal statement:

\begin{proposition}\label{UlmClassification}
	$\mbox{CSS}(\mbox{TAG}_1) \cong_{a\Delta_1} \mathcal{P}_{\aleph_1}(\omega_1)$.
\end{proposition}
\begin{proof}
	
	Since $\mathcal{P}_{\aleph_1}(\omega_1)^\omega \cong_{a\Delta_1} \mathcal{P}_{\aleph_1}(\omega_1)$, it suffices to show that each $\mbox{TAG}_{1, p} \cong_{a\Delta_1} \mathcal{P}_{\aleph_1}(\omega_1)$ (it will be clear from the proof that the reductions are uniform in $p$). We aim to apply Theorem~\ref{SBForSets}.
	
	Let $\langle, \rangle: \omega_1^2 \to \omega_1$ be a $ZFC^-$-absolute bijection (the standard pairing function works). Define $f: \mbox{CSS}(\mbox{TAG}_1) \to \mathcal{P}_{\aleph_1}(\omega_1)$ via $f(A) = \{\langle 0, \tau(A) \rangle, \langle 1, \mbox{rk}(A^{\tau(A)}) \rangle, \langle 2+ \omega \cdot \alpha + n, m^\alpha_n(A) \rangle: \alpha < \tau(A) \}$; the point is that we encode $(\tau(A), \mbox{rank}(A^{\tau(A)}), A_\alpha/\cong: \alpha < \tau(A))$. The image of $f$ is absolute to transitive models of $ZFC^-$, by a straightforward application of Theorem~\ref{Karp}.  
	
	For the reverse direction, Zippin \cite{Zippin} has proven that if $(C_\alpha: \alpha < \gamma)$ is a sequence of countable direct sums of cyclic groups where $\gamma< \omega_1$, and if for all $\alpha$ with $\alpha+1 < \gamma$ we have that $C_\alpha$ contains elements of arbitrarily high order, then there is a reduced countable $p$-torsion group $A$ such that $\tau(A) =\gamma$ and each $A_\alpha \cong C_\alpha$. Moreover, every reduced abelian $p$-group has an Ulm sequence of this form. This easily allows a reduction $g: \mathcal{P}_{\aleph_1}(\omega_1) \leq_{a\Delta_1} \mbox{CSS}(\mbox{TAG}_1)$, such that moreover the range of $g$ is the set of countable reduced $p$-groups, and hence is absolutely $\Delta_1$. 
	
	Thus we conclude by Theorem~\ref{SBForSets}.
\end{proof}

Henceforward we can identify $\mbox{CSS}(\mbox{TAG}_1)$ with $\mathcal{P}_{\aleph_1}(\omega_1)$. In particular $\| {\mbox{TAG}}_1 \| = | \mathcal{P}(\mbox{ON})| = \infty$, so this gives no upper bound on the complexity of ${\mbox{TAG}}_1$; nonetheless, Friedman and Stanley give a fairly simple proof that ${\mbox{TAG}}_1$ is not Borel complete (and the same proof shows it is not $\leq_{\HC}$-complete, in fact.) The need for a counting argument is more acute when we consider the jumps of ${\mbox{TAG}}_1$. We recall their definition from the introduction:

\begin{definition}
	For each $\alpha < \omega_1$, write ${\mbox{TAG}}_{1+\alpha} = J^\alpha({\mbox{TAG}}_1)$, the $\alpha$'th jump of torsion abelian groups.
\end{definition}

For all $\alpha \geq 1$, ${\mbox{TAG}}_\alpha \cong_{a\Delta_1} \mathcal{P}_{\aleph_1}^\alpha(\omega_1)$ by Theorem~\ref{countT} and the fact that $\omega \times \omega_1 \cong_{a\Delta_1} \omega_1$; thus we can identify $\CSS({\mbox{TAG}}_\alpha)$ with $\mathcal{P}^\alpha_{\aleph_1}(\omega_1)$. Thus we are identifying $\CSS(\mbox{TAG}_\alpha)_{\ptl}$ with $\mathcal{P}^\alpha(\mbox{ON})$. This is a small proper class, in the sense that each $|\mathcal{P}^\alpha(\mbox{ON}) \cap \mathbb{V}_{\lambda^+}| = \beth_\alpha(\lambda)$, which is less than the maximal possibility $|\mathbb{V}_{\lambda^+}| = \beth_{\lambda^+}$. Nonetheless, the simple proof that ${\mbox{TAG}}_1$ is not Borel complete does not carry through, and as far as we know the machinery we develop is necessary to conclude $\mbox{TAG}_\alpha$ is not Borel complete.

Our first attempt of directly counting $\mbox{CSS}(\Phi)_{\ptl} \cap \mathbb{V}_{\lambda^+}$ is problematic, because if $f: \Phi \leq_{a \Delta_1} \Psi$ we need not have that $f_{\ptl}$ carries $\mbox{CSS}(\Phi)_{\ptl} \cap \mathbb{V}_{\lambda^+}$ to $\mbox{CSS}(\Psi)_{\ptl} \cap \mathbb{V}_{\lambda^+}$. The following example demonstrates this:

\begin{example}
	Let $\mathcal{L}_0 = \{R_0\}$ and let $\mathcal{L}_1 = \{R_0, R_1\}$, where $R_0, R_1$ are binary relation symbols. Let $f: \mbox{Mod}(\mathcal{L}_1) \to \mbox{Mod}(\mathcal{L}_0)$ be the reduct map. Let $f_*: \CSS(\mathcal{L}_1) \to \CSS(\mathcal{L}_0)$ be the induced map on Scott sentences. Then for every cardinal $\lambda$ and for every $\kappa < \beth_{\lambda^+}$, there is some $\phi \in \mbox{CSS}(\mathcal{L}_1)_{\ptl} \cap \mathbb{V}_{\lambda^+}$, such that $(f_*)_{\ptl}(\phi) \not \in \mathbb{V}_{\kappa}$---in particular (choosing $\kappa = \lambda^+$), we can arrange $(f_*)_{\ptl}(\phi) \not \in \mathbb{V}_{\lambda^+}$. 
\end{example}
\begin{proof}
	Choose $\alpha < \lambda^+$ such that $\kappa^+ < \beth_{\alpha}$. We define an $\mathcal{L}_1$-structure $(M, R_0^M, R_1^M)$ as follows: let $(M, R_1^M) = (\mathbb{V}_{\alpha}, \in)$, and let $R_0^M$ be a well-ordering of $\mathbb{V}_\alpha$. Note that $(\mathbb{V}_\alpha, \in)$ is rigid and has Scott rank approximately $\alpha$, so $\mbox{css}(M, R_0^M, R_1^M) \in \mathbb{V}_{\lambda^+}$. On the other hand, $(M, R_0^M)$ is a well-ordering of length longer than $\kappa^+$, and so its canonical Scott sentence cannot be in $\mathbb{V}_\kappa$. 
\end{proof}

The idea for getting around this is to count $|\CSS(\Phi)_{\ptl} \cap A|$ for  $A \in \mathbb{V}_{\lambda^+}$ which are closed under $f_{\ptl}$ for various $f$.

\begin{definition}
	Let $\mathbb{F}$ be the set of all absolutely $\Delta_1$ $f$ such that persistently, $f: \HC \to \HC$. 
	
	Suppose $\overline{f} = (f_i: i < n) \in \mathbb{F}^{<\omega}$. Then say that $A$ is $\overline{f}$-closed if $A$ is a transitive set with $A^{<\omega} \subseteq A$, and $(f_i)_{\ptl}[A] \subseteq A$ for each $i < n$.
\end{definition}

If $f$ is an absolutely $\Delta_1$ map defined on some absolutely $\Delta_1$ $X \subseteq \HC$, we identify $f$ with $f' \in \mathbb{F}$ which is defined to be $\emptyset$ off of $X$.

The following simple lemma will be used implicitly henceforth:

\begin{lemma}\label{ThicknessPairing} Suppose $f_i: i < n$ is any sequence from $\mathbb{F}$. Define $f: \HC \to \HC$ to be $\prod_{i < n} f_i$, that is $f(a) = (f_i(a): i < n)$. Then $f \in \mathbb{F}$, and for every set $A$, we have that $A$ is $f$-closed if and only if $A$ is $\overline{f}$-closed.
\end{lemma}
\begin{proof}
	Clearly $f \in \mathbb{F}$. To finish, since we are requiring $A$ to be transitive and $A = A^{<\omega}$, we have that $((f_i)_{\ptl}(a): i < n) \in A$ if and only if each $(f_i)_{\ptl}(a) \in A$.
\end{proof}

The following fundamental observation will be the motivation for our definition of thickness:

\begin{lemma}\label{PreThickness}
	Suppose $X, Y$ are absolutely $\Delta_1$, such that for every $f \in \mathbb{F}$, there is an $f$-closed set $A$ with $|X_{\ptl} \cap A| > |Y_{\ptl} \cap A|$. Then $X \not \leq_{a \Delta_1} Y$.
\end{lemma}
\begin{proof}
	We prove the contrapositive. Suppose $f: X \leq_{a \Delta_1} Y$. As mentioned above, we view $f \in \mathbb{F}$ by defining $f(a) = \emptyset$ for $a \not \in X$. Suppose $A$ is $f$-closed.  Then $f_{\ptl}$ clearly witnesses that $|X_{\ptl} \cap A| \leq |Y_{\ptl} \cap A|$.
\end{proof}

\begin{definition}
	Suppose $X$ is absolutely $\Delta_1$. Suppose $\lambda$ is an infinite cardinal.
	
	Then define $\tau(X, \lambda)$, the {\em thickness of $X$ at $\lambda$}, to be the least cardinal $\kappa$ such that there is some $f \in \mathbb{F}$ such that $|X_{\ptl} \cap A| \leq \kappa$ for all $f$-closed $A \in \mathbb{V}_{\lambda^+}$. Alternatively, we have that $\tau(X, \lambda) > \kappa$ if and only if for every $f \in \mathbb{F}$, there is some $f$-closed $A \in \mathbb{V}_{\lambda^+}$ with $|X_{\ptl} \cap A| > \kappa$. 
	
	If $\Phi$ is a sentence of $\mathcal{L}_{\omega_1 \omega}$ then define $\tau(\Phi, \lambda) = \tau(\CSS(\Phi), \lambda)$.
\end{definition}

Some simple observations: $\tau(X, \lambda) \leq |X_{\ptl} \cap \mathbb{V}_{\lambda^+}| \leq \beth_{\lambda^+}$, and $\tau(X, \lambda)$ is monotone in $\lambda$, with $\lim_{\lambda \to \infty} \tau(X, \lambda) =\|X\|$.

The following theorem is a simple twist to the idea of Lemma~\ref{PreThickness}, just packaged in terms of the $\tau$ function.
\begin{theorem}\label{ThicknessInvariance}
	If $X_1 \leq_{a\Delta_1} X_2$, then $\tau(X_1, \lambda) \leq \tau(X_2, \lambda)$ for every infinite cardinal $\lambda$.
\end{theorem}
\begin{proof}
	Choose $f: X_1 \leq_{a\Delta_1} X_2$. Let $\lambda$ be given. Suppose towards a contradiction that $\tau(X_1, \lambda) > \tau(X_2, \lambda)= \kappa$. Choose $g \in \mathbb{F}$ witnessing that $\tau(X_2, \lambda) = \kappa$, that is, whenever $A \in \mathbb{V}_{\lambda^+}$ is $g$-closed, we have $|(X_2)_{\ptl} \cap A| \leq \kappa$. 
	
	By hypothesis (and Lemma~\ref{ThicknessPairing}), we can find some $(f, g)$-closed $A \in \mathbb{V}_{\lambda^+}$ such that $|(X_1)_{\ptl} \cap A| > \kappa$; by choice of $g$, $|(X_2)_{\ptl} \cap A| \leq \kappa$. But since $A$ is also $f$-closed, we have that $f_{\ptl}$ restricts to an injection from $(X_1)_{\ptl}\cap A$ to $(X_2)_{\ptl}\cap A$, a contradiction.
\end{proof}

The following theorem is also straightforward.

\begin{theorem}\label{ThicknessJumpsLowerBound}
	For all $\Phi, \lambda, \alpha$, if $\Phi$ has infinitely many countable models, then $\tau(J^\alpha(\Phi), \lambda) \leq \beth_\alpha(\tau(\Phi, \lambda))$.
\end{theorem}
\begin{proof}
	Write $\kappa = \tau(\Phi, \lambda)$; choose $f \in \mathbb{F}$ such that whenever $A \in \mathbb{V}_{\lambda^+}$ is $f$-closed, then $|\CSS(\Phi)_{\ptl} \cap A| \leq \kappa$. Then clearly also $|\CSS(J^\alpha(\Phi))_{\ptl} \cap A| \leq \beth_\alpha(\kappa)$ as desired.
\end{proof}

We do not know how to prove the reverse inequality in general, although we suspect that at least for $\lambda = \aleph_0$, it should be true. Instead we focus on special cases, where $\Phi$ is either some $\Phi_\alpha$ or some ${\mbox{TAG}}_\alpha$. Our task boils down to constructing thick transitive sets in $\mathbb{V}_{\lambda^+}$, as the following lemma indicates.

\begin{lemma} \label{ThicknessComp0}
	There is some $f \in \mathbb{F}$, such that for every $f$-closed $A$, $|\CSS(\mbox{Graphs})_{\ptl} \cap A| = |A|$, and for every $\alpha < \omega_1$, $|\mbox{CSS}(\Phi_\alpha)_{\ptl} \cap A| = |\mathcal{P}^\alpha(\omega) \cap A|$, and $|\mbox{CSS}({\mbox{TAG}}_\alpha)_{\ptl} \cap A| = |\mathcal{P}^\alpha(\mbox{ON}) \cap A|$. 
\end{lemma}
\begin{proof}
	We claim we can choose $f \in \mathbb{F}$ so as to encode reductions between $\mbox{Graphs}$ and $\HC$, between $\Phi_\alpha$ and $\mathcal{P}^\alpha_{\aleph_1}(\omega)$ for each $\alpha < \omega_1$, and between ${\mbox{TAG}}_\alpha$ and $\mathcal{P}^\alpha_{\aleph_1}(\omega_1)$ for each $\alpha < \omega_1$; and finally, the map sending $a$ to the foundation rank $\mbox{rnk}(a)$. Finding $f$ is not hard; note, for instance, that we can find some $f_0 \in \mathbb{F}$ such that persistently, for all $\alpha < \omega_1$, $f_0 \restriction_{\{\alpha\} \times \mbox{\small{CSS}}(\Phi_\alpha)}$ induces an absolutely $\Delta_1$-reduction from $\Phi_\alpha$ to $\mathcal{P}^\alpha_{\aleph_1}(\omega)$. $f$ will be a product of several such $f_i$'s.
	
	Then it is straightforward to see that $f$ works. For instance, suppose $A$ is $f$-closed, and either $\mbox{CSS}(\Phi_\alpha)_{\ptl} \cap A$ or else $\mathcal{P}^\alpha(\omega) \cap A$ is nonempty. Then $\alpha \in A$ since $A$ is closed under $\mbox{rnk}$, so $A$ will be $(g, h)$-closed, where $g, h$ are the $ZFC^-$-reductions between $\Phi_\alpha$ and $\mathcal{P}^\alpha(\omega)$ coded by $f$.
\end{proof}
The following definition is motivated by the above lemma.

\begin{definition}
	The infinite cardinal $\lambda$ \emph{admits thick sets} if for every $\alpha < \lambda^+$, and for every $f \in \mathbb{F}$, there is some $f$-closed $A \in \mathbb{V}_{\lambda^+}$, such that $|\mathcal{P}^\alpha(\lambda) \cap A| = \beth_\alpha(\lambda)$.
\end{definition}

Note that $|\mathcal{P}^\alpha(\lambda) \cap A| \leq |\mathcal{P}^\alpha(\lambda)| = \beth_\alpha(\lambda)$ always, so it suffices to assert $\geq$ in the above definition.

Over the next several sections, we prove the following.

\begin{theorem} \label{FatModelsExist1}
	Every regular strong limit cardinal admits thick sets. 
\end{theorem}

A regular strong limit is either $\aleph_0$ or inaccessible. So in particular, $\aleph_0$ admits thick sets.

We have the following:
\begin{corollary}\label{ThicknessComp}
	Suppose $\lambda$ is a regular strong limit. Then for every $\alpha < \omega_1$, $\tau(\Phi_\alpha, \lambda) = \beth_\alpha$, and $\tau({\mbox{TAG}}_\alpha, \lambda) = \beth_\alpha(\lambda)$. Also, if $\Phi$ is Borel complete then $\tau(\Phi, \lambda) = \beth_{\lambda^+}$. In particular, this happens whenever $\lambda$ is a regular strong limit.
\end{corollary}
\begin{proof}
	Choose $f$ as in Theorem \ref{ThicknessComp0}.
	
	For $\Phi_\alpha$, we will not actually need our hypothesis on $\lambda$: note that $\aleph_0$ is a regular strong limit, and hence admits thick sets. Then $f$ witnesses that $\tau(\Phi_\alpha, \aleph_0) = \beth_\alpha$: suppose $A \in \mathbb{V}_{\aleph_1}$ is $f$-closed. Then $|\mbox{CSS}(\Phi_\alpha)_{\ptl} \cap A| = |\mathcal{P}^\alpha(\omega) \cap A|$. This is always at most $\beth_\alpha$, but since $\aleph_0$ admits thick sets, for every $g \in \mathbb{F}$ we can also arrange that $A$ is $g$-closed and $|\mathcal{P}^\alpha(\omega) \cap A| = \beth_\alpha$.
	
	The rest is similar.
\end{proof}

We have the following immediate consequence; the case $\alpha = 1$ was proved by Friedman and Stanley in \cite{FS}, but for $\alpha > 1$, it is new that ${\mbox{TAG}}_\alpha$ is not Borel complete.

\begin{corollary}\label{UsAndTs}
	For all $1 \leq \alpha < \omega_1$, $\Phi_{\alpha+1} \not \leq_{B} {\mbox{TAG}}_\alpha$ (in fact $\Phi_{\alpha+1} \not \leq_{a\Delta_1} {\mbox{TAG}}_\alpha$).
\end{corollary}
\begin{proof}
	This is because $\tau(\Phi_{\alpha+1}, \aleph_0) = \beth_{\alpha+1} > \beth_\alpha = \tau({\mbox{TAG}}_\alpha, \aleph_0)$.
\end{proof}


\section{$f_{\ptl}$ is sufficiently absolute}\label{Technical1}
This is the first of three technical sections where we prove needed lemmas for Theorem~\ref{FatModelsExist1}. This section is the only place in our arguments where $\leq_{\HC}$ from \cite{URL} would not work; we need the additional absoluteness of $\leq_{a \Delta_1}$.

We begin with the following observation:

\begin{lemma}\label{AmenableTechnical}
Suppose $\lambda$ is a regular cardinal and $P \in H(\lambda)$ is a forcing notion. Suppose $G$ is $P$-generic over $\mathbb{V}$. Then $H(\lambda)[G] = H(\lambda)^{\mathbb{V}[G]}$.
\end{lemma}
\begin{proof}
Clearly, $H(\lambda)[G] \subseteq H(\lambda)^{\mathbb{V}[G]}$. Conversely, suppose $a \in H(\lambda)^{\mathbb{V}[G]}$; we need to find a name for $a$ in $H(\lambda)$. Let $b$ be the transitive closure of $a \cup \{a\}$. Let $\mbox{rnk}$ be foundation rank. Let $\gamma_* = \mbox{rnk}(a)< \lambda^+$, and choose a surjection $f: (\gamma_*+1) \times \lambda \to b$, such that for all $(\gamma, \alpha) \in \gamma_* \times \lambda$, $\mbox{rnk}(f(\gamma, \alpha)) \leq \gamma$, and such that $f(\gamma_*, 0) = a$.

Choose $P$-names $\dot{a}, \dot{b}, \dot{f}$ and some $p \in G$ such that $\mbox{val}(\dot{a}, G) = a$, $\mbox{val}(\dot{b}, G) = b$, and $\mbox{val}(\dot{f}, G) = f$, and such that $p$ forces the preceding holds. 

The remainder of the argument takes place in $\mathbb{V}$. 

By induction on $\gamma \leq \gamma_*$, define $P$-names $(\dot{c}_{\gamma, \alpha}: \alpha < \lambda)$. Namely, $\dot{c}_{\gamma, \alpha}$ has domain $\{\dot{c}_{\gamma', \beta}: \beta < \lambda, \gamma' < \gamma\}$, and we want each $\dot{c}_{\gamma, \alpha}(\dot{c}_{\gamma', \beta}) = \|\dot{f}(\gamma', \beta) \in \dot{f}(\gamma, \alpha)\|_{\mathcal{B}(P)}$, an element of $\mathcal{B}(P)$. (Here $\mathcal{B}(P)$ is the Boolean-algebra completion of $P$, see \cite{Jech}. We can define $\mathcal{B}(P)$ as a subset of the powerset of $P$, and hence as a subset of $H(\lambda)$.) Then $p \Vdash \dot{c}_{\gamma_*, 0} = \dot{a}$, and $\dot{c}_{\gamma_*, 0} \in H(\lambda)$, so we are done.
\end{proof}
In particular, in the above situation, $\HC^{H(\lambda)[G]} = \HC^{\mathbb{V}[G]}$.

\begin{lemma}\label{ProductForcing}
    Suppose $V \models ZFC^-$ and $P \in V$ and $G \times G'$ is $P \times P$-generic over $V$. Then $V[G] \cap V[G'] = V$.
\end{lemma}
\begin{proof}
    When $V \models ZFC$ this is Corollary 2.3 of \cite{KaplanShelah}. The proof for $ZFC^-$ is the same.
\end{proof}

\begin{lemma}\label{ptlAbsolute}
    Suppose $f$ is absolutely $\Delta_1$ and persistently, $f: \HC \to \HC$; say $f$ is over the parameter $a_* \in \HC$. Suppose $\lambda$ is a regular uncountable cardinal, and $V \preceq H(\lambda)$ is a transitive elementary submodel containing $a_*$. Finally suppose $P \in V$ is a forcing notion, and $H$ is $P$-generic over $V$, living in some forcing extesion $\mathbb{V}[G]$ of $\mathbb{V}$.

    Consider $f_{\ptl}^{\mathbb{V}[G]}$,  namely $f_{\ptl}$ as computed in $\mathbb{V}[G]$ (so $f_{\ptl} = f_{\ptl}^{\mathbb{V}[G]} \restriction_{\mathbb{V}}$). Then we have: $V[H]$ is closed under $f_{\ptl}^{\mathbb{V}[G]}$.
\end{lemma}
\begin{proof}
    Let $\phi(x)$ and $\psi(x)$ be the $\Sigma_1$ and $\Pi_1$ definitions of $f$, so we are assuming $\phi(x)$ and $\psi(x)$ are over $a_*$.
    
    Let $\Gamma$ be the fragment of set theory over the parameter $a_*$ asserting of its putative model $W$ that $W \models ZFC^-$ and, for every forcing extension $W[H]$, $\phi(x)$ and $\psi(x)$ coincide on $\HC^{W[H]}$ and define the graph of a function with domain $\HC^{W[H]}$. This is a first-order assertion of $W$ by Lemma~\ref{ForcingLemma1}.

    Our hypotheses on $\phi(x), \psi(x), f$ imply that $\mathbb{V} \models \Gamma$. Lemma~\ref{AmenableTechnical} implies then that $H(\lambda) \models \Gamma$. By elementarity, then, $V \models \Gamma$. Since every forcing extension of $V[H]$ is itself a forcing extension of $V$, we get that $V[H] \models \Gamma$.  

So after passing from $\mathbb{V}$ to $\mathbb{V}[G]$, it suffices to show that if $W \subset \mathbb{V}$ is a transitive model of $\Gamma$, then $W$ is closed under $f_{\ptl}$. Let $a \in W$. Let $Q \in W$ be a forcing notion which makes $a$ hereditarily countable. Let $K$ be $Q$-generic over $\mathbb{V}$. Now by definition of $\Gamma$ we have that there is some $b \in \HC^{W[K]}$ with $\HC^{W[K]} \models \phi(a, b)$. Thus $\HC^{\mathbb{V}[K]} \models \phi(a, b)$. By definition of ptl, this $b$ must be equal to $f_{\ptl}(a)$ as computed in $\mathbb{V}$. Thus $f_{\ptl}(a) \in W[K]$. But then let $K'$ be such that $K \times K'$ is $P \times P$-generic over $W$; by a symmetric argument we get that $f_{\ptl}(a) \in W[K']$, hence it lies in $W[K] \cap W[K'] = W$ by Lemma~\ref{ProductForcing}.
\end{proof}

\section{Independence Lemmas}\label{DensityLemmas}

This is a second technical section, in which we prove some needed facts for Theorem~\ref{FatModelsExist1}. The idea there is that given some transitive $V \models ZFC^-$ with $\alpha, \lambda \in V$, we wish to construct some large $A \in \mathcal{P}^\alpha(\lambda)$, such that $A$ lies in a forcing extension of $V$; roughly speaking, this  means that $A$ does not code anything terrible about $V$, e.g. a bijection between $V$ and $\omega$.

The following definition and theorem form the combinatorial underpinnings of our approach.

\begin{definition}
	Suppose $Y \subseteq \mathcal{P}(X)$. By a {\em finite Boolean combinatio}n from $Y$ we mean a set of the form $a_0 \cap \ldots \cap a_{n-1} \cap (X \backslash b_0) \cap \ldots \cap (X \backslash b_{m-1})$, for some $a_0, \ldots, a_{n-1}, b_0, \ldots, b_{m-1}$ from $Y$ with each $a_i \not= b_j$. $Y$ is {\em independent over} $X$ if and only if each finite Boolean combination from $Y$ is nonempty.
\end{definition}

The following is a special case of a theorem of Engelking and Karlowicz \cite{Partitions}.
\begin{lemma}\label{Comb1}
	Suppose $\kappa$ is an infinite cardinal. Then there is $Y \subseteq \mathcal{P}(\kappa)$ which is independent over $\kappa$ with $|Y| = 2^{\kappa}$.
\end{lemma}
\begin{proof}Choose $D \subseteq 2^{2^\kappa}$ of size $\kappa$ such that for each $s \in [\kappa]^{<\aleph_0}$ and each $\overline{f}: 2^s \to 2$, there is some $F \in D$ such that for all $g \in 2^\kappa$, $F(g) = \overline{f}(g \restriction_s)$.
	
	Write $D = \{F_\alpha: \alpha < \kappa\}$. For each $f \in 2^\kappa$ put $Y_f = \{\alpha < \kappa: F_\alpha(f) = 1\} \subseteq \kappa$. Let $Y = \{Y_f: f \in 2^\kappa\}$. We claim this works; clearly $|Y| = 2^\kappa$. Moreover, given $(f_i: i < i_*), (g_j: j < j_*)$ sequences of distinct elements from $2^\kappa$ with $i_*, j_* < \omega$, we can choose $s \in [\kappa]^{<\aleph_0}$ such that $f_i \restriction_s$, $g_j \restriction_s$ are all distinct. Then choose $\overline{f}: 2^s \to 2$ so that each $\overline{f}(f_i \restriction_s) = 1$, each $\overline{f}(g_j \restriction_s) = 0$. By choice of $D$ applied to $\overline{f}$, there is some $\alpha < \kappa$ such that $F_\alpha(f_i) = 1$ and $F_\alpha(g_j) = 0$ for $i < i_*, j < j_*$; i.e. $\alpha \in Y_{f_i}$ for $i < i_*$ and $\alpha \not \in Y_{g_j}$ for $j < j_*$. This suffices to show independence.
\end{proof}

We now wish to strengthen this. Some definitions will explain what we want:

\begin{definition}
	\begin{itemize}
		\item Suppose $X$ is a topological space. Then $X$ is $\kappa$-{\em nice} if $X$ has a basis of cardinality (at most) $\kappa$, and every nonempty open subset of $X$ has size $\kappa$. (In particular, $|X| = \kappa$.) 
		\item If $X$ is a topological space and $D \subseteq X$, then say that $D$ is $\kappa$-{\em dense} in $X$ if whenever $\mathcal{O} \subseteq X$ is open nonempty, then $|D \cap \mathcal{O}| \geq \kappa$.

		\item Suppose $X$ is a topological space and $Y \subseteq \mathcal{P}(X)$. Then $Y$ is {\em densely independent} if every finite Boolean combination from $Y$ is dense in $X$. Equivalently, for each nonempty open subset $\mathcal{O}$ of $X$, every finite Boolean combination from $Y$ intersects $\mathcal{O}$. 
		
	\end{itemize}
\end{definition}
A routine diagonalizing argument shows that if $X$ is $\kappa$-nice, then we can write $X$ as the disjoint union of $(X_\alpha: \alpha < \kappa)$, where each $X_\alpha$ is dense in $X$ and $\kappa$-nice.

Now we massage Lemma~\ref{Comb1} into the form we want.

\begin{theorem}\label{Comb}
	Suppose $X$ is $\kappa$-nice. Give $\mathcal{P}(X)$ the finite support topology. Then there is a sequence $(Y_\delta: \delta <  2^\kappa)$ of disjoint subsets of $\mathcal{P}(X)$ such that each $Y_\delta$ is $2^\kappa$-dense in $\mathcal{P}(X)$, and $\bigcup_\delta Y_\delta$ is densely independent over $X$.
\end{theorem}

\begin{proof}
	Write $X$ as the disjoint union $(X_\alpha: \alpha < \kappa)$, where each $X_\alpha$ is $\kappa$-nice, and dense in $X$. Let $(\mathcal{O}_\alpha: \alpha < \kappa)$ be a basis of $X$, and let $(\mathcal{U}_\alpha: \alpha < \kappa)$ be an enumeration of $(\mathcal{O}_\alpha: \alpha < \kappa)$ in which each $\mathcal{O}_\alpha$ occurs $\kappa$-often.

	By Lemma~\ref{Comb1}, for each $\alpha < \kappa$ we can choose $Y'_\alpha \subseteq \mathcal{P}(\mathcal{U}_\alpha \cap X_\alpha)$ with $|Y'_\alpha| = 2^\kappa$, such that $Y'_\alpha$ is independent over $\mathcal{U}_\alpha \cap X_\alpha$. Enumerate $Y'_\alpha = \{b^{\gamma, \alpha}_{\delta}: \gamma, \delta < 2^\kappa\}$. For each $\gamma, \delta < 2^\kappa$, let $b^\gamma_\delta = \bigcup_{\alpha < \kappa} b^{\gamma, \alpha}_{\delta}$ and let $Y' = \{b^\gamma_\delta: \gamma, \delta < 2^\kappa\}$. Then each finite Boolean combination from $Y'$ intersects each $\mathcal{U}_\alpha \cap X_\alpha$ (this is the property we need of $Y'$ going forth).
	
	Let $(s_\gamma, t_\gamma: \gamma< 2^\kappa)$ enumerate all pairs of disjoint finite subsets of $X$, such that each pair occurs $2^\kappa$-many times. Let $E_\gamma \in [\kappa]^{<\aleph_0}$ be large enough so that $s_\gamma \cup t_\gamma \subseteq \bigcup_{\alpha \in E_\gamma} X_\alpha$. For each $\delta, \gamma < 2^\kappa$ let $c_{\delta}^\gamma = (b_{\delta}^\gamma \cup s_\gamma) \backslash t_\gamma$. Let $Y_{\delta} = \{c^\gamma_{\delta}: \gamma < 2^{\kappa}\}$. 
	
	We claim that $(Y_\delta: \delta < 2^\kappa)$ works. Each $Y_\delta$ is $2^\kappa$-dense in $\mathcal{P}(X)$ by choice of $(s_\gamma, t_\gamma: \gamma < 2^\kappa)$. 
	
	We check that $Y_\delta \cap Y_{\delta'} = \emptyset$ whenever $\delta \not= \delta'$. Indeed, suppose $c = c_{\delta}^\gamma \in Y_{\delta}$ and $c' = c_{\delta'}^{\gamma'} \in Y_{\delta'}$. Write $E = E_{\gamma} \cup E_{\gamma'}$. Choose $\alpha \in \kappa \backslash E$. Then $c \cap X_\alpha = b^\gamma_{\delta} \cap X_\alpha$ and $c' \cap X_\alpha = b^{\gamma'}_{\delta'} \cap X_\alpha$. Since $(b^\gamma_{\delta} \backslash b^{\gamma'}_{\delta'}) \cap X_\alpha  \not= \emptyset$ we conclude $c \not= c'$.
	
	Finally, suppose $d := \bigcap_{i < i_*} (c_i)^{\pm 1}$ is a finite Boolean combination from $\bigcup_\delta Y_\delta$ and suppose $\beta < \kappa$; we show that $d \cap \mathcal{O}_\beta$ is nonempty. For each $i < i_*$, choose $\delta_i,  \gamma_i$ with $c_i = c_{\delta_i}^{\gamma_i}$. Let $E = \bigcup_{i < i_*} E_{\gamma_i}$. Choose $\alpha \in \kappa \backslash E$ such that $\mathcal{U}_\alpha = \mathcal{O}_\beta$. Note that each $c_i \cap X_\alpha = b^{\gamma_i}_{\delta_i} \cap X_\alpha$. By choice of $Y'$, $\mathcal{U}_\alpha \cap X_\alpha \cap \bigcap_{i < i_*} (b^{\gamma_i}_{\delta_i})^{\pm 1}\not= \emptyset$. Hence $\mathcal{U_\alpha} \cap d \not= \emptyset$, and $\bigcup_\delta Y_\delta$ is densely independent.
\end{proof}

We now apply this towards constructing large subsets of iterated powersets that are generic in some sense. The following definition is complicated somewhat by the fact that we want the $X_\alpha$'s to be disjoint.

\begin{definition}
    Let $\mathcal{P}_{\not= \emptyset}(X)$ denote the class of nonempty subsets of $X$.
    
    Suppose $\alpha_* \geq 1$. Then say that $\overline{X} := (X_\alpha: \alpha < \alpha_*)$ is an $\alpha_*$-{\em powerset sequence} if 

    \begin{enumerate}
        \item $X_0$ is a nonempty set, such that all elements of $X_0$ have the same foundation rank;
        \item For each $\beta$ with $\beta+1 < \alpha_*$, $X_{\beta+1} \subseteq \mathcal{P}_{\not= \emptyset}(X_\beta)$;
        \item For limit $\delta$, $X_{\delta} \subseteq \prod_{\beta < \delta} \mathcal{P}_{\not= \emptyset}(X_\beta)$;
        \item For limit $\delta$, for all $a, b \in X_\delta$, the set of all $\beta$ with $a(\beta) = b(\beta)$ is an ordinal.
    \end{enumerate}
    If $\lambda$ is a regular cardinal, then say that $\overline{X}$ is a $(\lambda, \alpha_*)$-powerset sequence if additionally $X_0 \subseteq \mathcal{P}(\lambda)$.
   \end{definition}

   \begin{lemma}
       Suppose $\overline{X}$ is an $\alpha_*$-powerset sequence. Then the $X_\alpha$'s are all disjoint.
   \end{lemma}
   \begin{proof}
       Let $\gamma$ be the foundation rank of some or any element of $X_0$. Then for all $\alpha < \alpha_*$, any element of $X_\alpha$ has foundation rank $\gamma + \alpha$, so we are done.
   \end{proof}

    We introduce some topologies. We warn the reader that $\mathcal{P}(\lambda)$ carries the $<\lambda$-support topology, while $\mathcal{P}(X_\alpha)$ carries the finite support product topology. 
    
    \begin{definition}
    Suppose $\overline{X}$ is a $(\lambda, \alpha_*)$ powerset-sequence. Then for $\alpha < \alpha_*$ we define a topology $\tau_\alpha$ on $X_\alpha$. Namely, let $\tau_0$ be the $<\lambda$-support product topology coming from considering $X_0 \subseteq 2^\lambda$. For all $\beta$ with $\beta+1 < \alpha_*$, $\tau_{\beta+1}$ is the finite-support product topology from considering $X_{\beta+1} \subseteq 2^{X_\beta}$. For all limit $\delta$, $\tau_\delta$ is the tree topology whose basic open sets are $\mathcal{O}_{a, \beta}$ for $a \in X_\delta$ and $\beta < \delta$, where $\mathcal{O}_{a, \beta} := \{b \in X_\delta: b(\beta) = a(\beta)\}$.

    Let $\Omega_{\alpha_*}$ denote the set of all pairs $\beta < \alpha < \alpha_*$ such that either $\alpha = \beta+1$ or else $\alpha$ is a limit ordinal. Suppose $\overline{X}$ is a $(\lambda, \alpha_*)$-powerset sequence and $(\beta, \alpha) \in \Omega_{\alpha_*}$. We define $B_{\beta \alpha} \subseteq \mathcal{P}(X_\beta)$ as follows: if $\alpha= \beta+1$, then $B_{\beta \alpha} := X_\alpha$. If $\alpha = \delta$ is a limit, then $B_{\beta \delta} := \{a(\beta): a \in X_\delta\}$. Given $\beta < \beta' <\delta$ and $b \in B_{\beta \delta}, b' \in B_{\beta' \delta}$, say that $b \prec b'$ if there is some $a \in X_\delta$ with $a(\beta) = b, a(\beta') = b'$. This turns $\bigcup_{\beta < \delta} B_{\beta \delta}$ into a tree of height $\delta$; its $\beta$'th level is $B_{\beta \delta}$. 

   Suppose that $\overline{X}$ is a $(\lambda, \alpha_*)$-powerset sequence. Then say that $\overline{X}$ is {\em rich} if:

    \begin{enumerate}
        \item $X_0$ is $2^\lambda$-dense in $\mathcal{P}(\lambda)$;
        \item For each $\beta < \alpha_*$, $(B_{\beta \alpha}: (\beta, \alpha) \in \Omega_{\alpha_*})$ are pairwise disjoint, and their union is densely independent over $X_\beta$;
        \item Whenever $\beta+1 < \alpha_*$, $B_{\beta \beta+1}$ is $2^{|X_\beta|}$-dense in $\mathcal{P}(X_\beta)$ (with the finite support product topology);
        \item Whenever $\beta < \delta< \alpha_*$ with $\delta$ limit, and whenever $b_{\beta'}: \beta' < \beta$ is a sequence such that $b_{\beta'} \in B_{\beta' \delta}$ and $\beta' < \beta''$ implies $b_{\beta'} \prec b_{\beta''}$, then the set of all $b \in B_{\beta \delta}$ with each $b_\beta \prec b'$ is $2^{|X_\beta|}$-dense in $\mathcal{P}(X_\beta)$ (with finite support product topology);
        \item Whenever $\delta < \alpha_*$ is a limit ordinal, and whenever $(b_\beta: \beta < \delta)$ is a branch through $\bigcup\{B_{\beta \delta}: \beta < \delta\}$, then there is $a \in X_\delta$ with each $a(\beta) = b_\beta$.
    \end{enumerate}
    \end{definition}

    \begin{lemma}
        Suppose $\overline{X}$ is a rich $(\lambda, \alpha_*)$-powerset sequence. Then each $|X_\alpha| = \beth_{\alpha+1}(\lambda)$.
    \end{lemma}
    \begin{proof}
        Straightforward induction.
    \end{proof}
    
     \begin{lemma}\label{RichExists}
        Suppose $\lambda$ is a regular cardinal and $\alpha_* < \lambda^+$. Suppose $X_0 \subseteq \mathcal{P}(\lambda)$ is $2^\lambda$-dense in $\mathcal{P}(\lambda)$ (with the $<\lambda$-support product topology), and each $a \in X_0$ is cofinal in $\lambda$. Then there is a rich $(\lambda, \alpha_*)$-powerset sequence $\overline{X}$ extending $X_0$. 
     \end{lemma}
     \begin{proof}
        All elements of $X_0$ have the same foundation rank, namely $\lambda$, since they are all cofinal in $\lambda$.

        We build $\overline{X}$ inductively, but we will need extra data on the side. At stage $\alpha$, we will have constructed a rich $(\lambda, \alpha)$-powerset sequence $(X_\beta: \beta < \alpha)$, along with $(Y_{\beta, \gamma, i}: \beta < \alpha, \gamma < \alpha_*, i < 2^{|X_\beta|})$ with each $Y_{\beta, \gamma, i} \subseteq \mathcal{P}(X_\beta)$ and each $Y_{\beta, \gamma, i}$ is $2^{|X_\beta|}$ dense in $\mathcal{P}(X_\beta)$. We will further have that for all $\beta$, $(Y_{\beta, \gamma, i}: \gamma < \gamma_*, i < 2^{|X_\beta|})$ is pairwise disjoint and its union is densely independent over $X_\beta$. Finally we will have arranged that for all $\beta+1 < \alpha$, $B_{\beta, \beta+1} = Y_{\beta, \beta+1, 0}$ and for all $\beta < \delta < \alpha$ with $\delta$ limit, $B_{\beta \delta} = \bigcup\{Y_{\beta, \delta, i}: i < 2^{|X_\beta|}\}$.
        
        At stage $\alpha = 0$ we just need to produce $(Y_{0, \gamma, i}: \gamma < \alpha_*, i < 2^{|X_0|})$, which is possible by Theorem~\ref{Comb} and the fact that $X_0$ is $2^\lambda$-nice, since $X_0$ is $2^\lambda$-dense in $\mathcal{P}(\lambda)$.
        
        Suppose we are at stage $\alpha < \alpha_*$. We need to define $X_\alpha$ and $Y_{\alpha, \gamma, i}: \gamma < \alpha_*, i < 2^{|X_\alpha|}$.
        
        If $\alpha = \beta+1$ then put $X_\alpha = Y_{\beta, \beta+1, 0}$. Then since $Y_{\beta, \beta+1, 0}$ is $2^{|X_\beta|}$-dense in $\mathcal{P}(X_\beta)$, $X_\alpha$ is $|X_\alpha|$-nice. Thus, we can apply Theorem~\ref{Comb} to get $Y_{\alpha, \gamma, i}: \gamma < \alpha_*, i < 2^{|X_\alpha|}$.

        So suppose instead $\alpha = \delta$ is a limit. Let $B_{\beta \delta}$ denote $\bigcup \{Y_{\beta \delta i}: i < 2^{|X_\beta|}\}$. We turn $B := \bigcup_{\beta < \delta} B_{\beta \delta}$ into a tree under $\prec$, satisfying the following conditions:

        \begin{itemize}
            \item The $\beta$'th level of $B$ is $B_{\beta \delta}$;
            \item For all $\beta < \delta$ and for all $\prec$-branches $(b_{\beta'}: \beta' < \beta)$ with $b_{\beta'} \in B_{\beta' \delta}$, there is some $i < 2^{|X_\beta|}$ such that every element of $Y_{\beta \delta i}$ extends $(b_{\beta'}: \beta' < \beta)$ under $\prec$.
        \end{itemize}

        To see this is possible, we just need to show $2^{|X_\beta|} \geq \prod_{\beta' < \beta} 2^{|X_{\beta'}|}$, which is true, using $\beta < \alpha_* < \lambda^+ \leq 2^{|X_0|} \leq |X_\beta|$.

        Now let $X_\delta$ be the set of all $\prec$-branches through $B$, i.e. all $b \in \prod_{\beta < \delta} B_{\beta \delta}$ such that $b(\beta) \prec b(\beta')$ for all $\beta < \beta' < \delta$. Some easy cardinal arithmetic shows that $|X_\delta| = \prod_{\beta < \delta} 2^{|X_\beta|}$ and that $X_\delta$ is $|X_\delta|$-nice, so we can apply Theorem~\ref{Comb} to find $(Y_{\delta, \gamma, i}: \gamma< \alpha_*, i < 2^{|X_\delta|})$ and continue.
     \end{proof}

    \begin{lemma}\label{PsiCat2}
        There exists $f \in \mathbb{F}$ such that the following holds. Suppose $A$ is $f$-closed (i.e. transitive, closed under $f_{\ptl}$, and closed under finite sequences). Suppose $\lambda \in A$ is a regular cardinal, and $\overline{X} \in A$ is a rich $(\lambda, \alpha_*)$-powerset sequence. Then $|\mathcal{P}^{\alpha_*}(\lambda) \cap A| = \beth_{\alpha_*}(\lambda)$.
    \end{lemma}
    \begin{proof}
        We must first explain how $f$ acts on hereditarily countable inputs; then $f_{\ptl}$ will be checked to have the required properties. Let $Y$ denote the class of all hereditarily countable $\overline{X}$ such that for some $\alpha_* < \omega_1$, $\overline{X}$ is an $\alpha_*$-powerset sequence with $X_0 \subseteq \mathcal{P}_{\aleph_1}(\omega_1)$. Then $Y$ is absolutely $\Delta_1$ and $Y_{\ptl}$ contains every $(\lambda, \alpha_*)$-powerset sequence.
    
        Suppose that $\overline{X} \in Y$. Write $\gamma := \bigcup X_0$. We inductively define a sequence $(F_\alpha: 0 < \alpha \leq \alpha_*)$ such that each $F_{\alpha+1}$ is an injection of $X_\alpha$ into $\mathcal{P}^{\alpha+1}_{\aleph_1}(\gamma)$, and such that for limit $\delta$, $F_\delta$ is an injection from $\bigcup\{X_\alpha: \alpha < \delta\}$ into $\mathcal{P}^\delta_{\aleph_1}(\gamma)$. 
        
        Let $F_1$ be the identity inclusion of $X_0$ into $\mathcal{P}_{\aleph_1}(\omega_1)$. For limit $\delta$, recall that $\mathcal{P}^\delta_{\aleph_1}(\gamma)$ is the disjoint union $\bigcup_{\alpha < \delta} \{\alpha\} \times \mathcal{P}^{\alpha}_{\aleph_1}(\gamma)$, so for $a \in X_\alpha$ with $\alpha < \delta$ let $F_\delta(a) = (\alpha+1, F_{\alpha+1}(a))$.

        Suppose $\alpha = \beta+1$ where $\beta$ is a successor ordinal, say $\beta = \beta'+1$. We want to define an injection $F_\alpha: X_\beta \to \mathcal{P}^\alpha_{\aleph_1}(\gamma)$. Now $X_\beta \subseteq \mathcal{P}_{\aleph_1}(X_{\beta'})$ so put $F_\alpha(a) = \{F_\beta(b): b \in a \}$ for each $a \in X_\beta$. 
        
        Finally suppose $\alpha = \delta+1$ for some limit $\delta$. Then given $a \in X_\delta$, we have $a \in \prod_{\beta < \delta} \mathcal{P}_{\aleph_1}(X_\beta)$. Let $F_{\delta+1}(a) = \{F_\delta(b): b \in \bigcup_{\beta < \delta} a(\beta)\}$.

        Define $f(\overline{X}) = F_{\alpha_*}$. Then easily $f$ is an absolutely $\Delta_1$ map, such that persistently, for any $\alpha_*$-powerset sequence $\overline{X} \in Y$ with $X_0 \subseteq \mathcal{P}(\gamma)$, if $\alpha_*$ is a limit ordinal then $f(\overline{X})$ is an injection from $\overline{X}$ into $\mathcal{P}^{\alpha_*}_{\aleph_1}(\gamma)$, and if $\alpha_* = \alpha+1$ then $f(\overline{X})$ is an injection from $X_{\alpha}$ into $\mathcal{P}^{\alpha_*}_{\aleph_1}(\gamma)$.

        Thus $f \in \mathbb{F}$ is as desired.
    \end{proof}

\section{Powerset Sequences as Structures} \label{TechnicalThree}

   We shall need to view the powerset sequences constructed in the last section as structures in a suitable sense. We explain how.
   

\begin{definition}
	Suppose $\alpha_* \geq 1$ is an ordinal. Recall that $\Omega_{\alpha_*}$ is the set of all pairs $(\beta, \delta)$ where $\beta < \delta < \alpha_*$ and $\delta$ is a limit ordinal, along with all pairs $(\beta, \beta+1)$ where $\beta+1 < \alpha_*$.

We aim to define a sentence $\Psi_{\alpha_*}$ describing certain $\alpha_*$-powerset sequences. We first describe a universal sentence $\Psi_{\alpha_*}^{\forall}$; $\Psi_{\alpha_*}$ will then describe sufficiently generic models of $\Psi_{\alpha_*}^{\forall}$.
	
	Let the language of $\Psi_{\alpha_*}$ consist of sorts $(U_\alpha: \alpha < \alpha_*)$, a binary relation $R$, and binary relations $E_{\beta, \delta} \subseteq U_\delta \times U_\delta$ for each $(\beta, \delta) \in \Omega_{\alpha_*}$. 
	
	Let $\Psi_{\alpha_*}^{\forall}$ be the universal sentence of $\mathcal{L}_{|\alpha_*|^+ \omega}$ asserting: 
	
	\begin{enumerate}
		\item $(U_\alpha: \alpha < \alpha_*)$ are disjoint and partition the universe.
		\item $R \subseteq \bigcup_{(\beta, \alpha) \in \Omega_{\alpha_*}} U_\beta \times U_\alpha$.
		\item For all $(\beta, \alpha) \in \Omega_{\alpha_*}$ and for all $a, b \in U_\alpha$, if $a E_{\beta \alpha} b$ then for every $c \in U_\beta$, $c R a$ if and only if $c R b$.
		\item For all $\alpha < \alpha_*$ successor ordinal, say $\alpha = \beta+1$, we have $E_{\beta \alpha}$ is discrete.
            \item For all $\delta < \alpha_*$ limit and for all distinct $a, b \in U_\delta$, there is some $\beta < \delta$ such that for all $ \beta' < \delta$, $a E_{\beta' \delta} b$ if and only if $\beta' < \beta$. We write $a \Delta b = \beta$ to indicate this.
	\end{enumerate}

 Let $\Psi_{\alpha_*}$ assert additionally of its model $M$ that given any $\overline{a} \in M^{<\omega}$ and any finite $N \models \Psi_{\alpha_*}^{\forall}$ extending $\overline{a}$, there is some isomorphic embedding of $N$ into $M$ over $\overline{a}$. 
\end{definition}
Some simple remarks:
\begin{remark}
In the definition of $\Psi_{\alpha_*}$, it is enough to consider the case where $N$ is a one-point extension of $\overline{a}$.

    Also, suppose $M \models \Psi_{\alpha_*}$. Then each $U_\alpha^M$ is infinite, and further, for all $(\beta, \alpha) \in \Omega_{\alpha_*}$ and for all $a, b \in U_\alpha$, $a E_{\beta \alpha} b$ of and only if for every $c \in U_\beta$, $c R a$ if and only if $c R b$.
\end{remark}

 \begin{lemma}\label{ExtraConditions}
     Suppose $M \models \Psi_{\alpha_*}^{\forall}$. Then $M \models \Psi_{\alpha_*}$ if and only if the following conditions both hold.
 \begin{enumerate}
		\item Suppose $\alpha < \alpha_*$ is not a limit ordinal. Suppose $u_0, u_1$ are disjoint finite subsets of $\bigcup\{U_\beta: (\beta, \alpha) \in \Omega_{\alpha_*}\}$ and suppose $v_0, v_1$ are disjoint finite subsets of $\bigcup\{U_\gamma: (\alpha, \gamma) \in \Omega_{\alpha_*}\}$, satisfying that for all $\gamma$ with $(\alpha, \gamma) \in \Omega_{\alpha_*}$ and for all $c_0, c_1$ with $c_i \in v_i \cap U_\gamma$, $c_0$ is not $E_{\alpha, \gamma}$-related to $c_1$. Then there are infinitely many $a \in U_\alpha$ such that $b R a$ for each $b \in u_0$, and $\lnot (b R a)$ for each $b \in u_1$, and $a R c$ for each $c \in v_0$, and $\lnot (a R c)$ for each $c \in v_1$.
		
		\item Suppose $\delta < \alpha_*$ is a limit, $\beta_* < \delta$, $d \in U_\delta$, $u_0, u_1$ are disjoint finite subsets of $\bigcup\{U_\beta: \beta_* \leq \beta < \delta\}$, and $v_0, v_1$ are disjoint finite subsets of $\bigcup\{U_\gamma: (\delta, \gamma) \in \Omega_{\alpha_*}\}$, satisfying that for all $\gamma$ with $(\delta, \gamma) \in \Omega_{\alpha_*}$ and for all $c_0, c_1$ with $c_i \in v_i \cap U_\gamma$, $c_0$ is not $E_{\delta, \gamma}$-related to $c_1$. Then there are infinitely many $E_{\beta_* \delta}$-inequivalent $a \in U_\delta$ such that $a \Delta d = \beta_*$, and $b R a$ for all $b \in u_0$, and $\lnot b R a$ for all $b \in u_1$, and $a R c$ for each $c \in v_0$, and $\lnot (a R c)$ for each $c \in v_1$.
	\end{enumerate}
\end{lemma}
\begin{proof}
Suppose first that $M \models \Psi_{\alpha_*}$; we show $M$ satisfies the second condition (the first is similar). Suppose $\delta, \beta_*, d, u_0, u_1, v_0, v_1$ are given. Let $\overline{a}$ be a list of $d, u_0, u_1, v_0, v_1$. Let $N$ consist of $\overline{a}$ and infinitely many new points $e_n: n < \omega$, which are all $E_{\beta_* \delta}$-inequivalent, and each $e_n \Delta d = \beta_*$, and $a R e_n$ for all $a \in u_0$, and $\lnot a R e_n$ for all $a \in u_1$, and $e_n R c$ for all $c \in v_0$, and $\lnot e_n R c$ for all $c \in v_1$. This specifies the isomorphism type of $N$, and clearly $N \models \Psi_{\alpha_*}^{\forall}$. Since every finite substructure of $N$ containing $\overline{a}$ can be embedded into $M$, we get that the second condition holds.

Suppose next that $M \models \Psi_{\alpha_*}^{\forall}$ and also satisfies the two conditions above. We show that $M \models \Psi_{\alpha_*}$. Suppose $\overline{a} \in M$ is a finite tuple and suppose $N \models \Psi^{\alpha_*}$ is a one-point extension of $\overline{a}$, say the new point is $e$. We handle the case where $e \in U_\delta$ for some limit $\delta$; the case where $e \in U_\alpha$ for $\alpha$ nonlimit is similar. Let $\beta_*$ be the maximum of $e \Delta d$ for $d \in U_\delta \cap \overline{a}$; let $d \in U_\delta \cap \overline{a}$ achieve this maximum. Let $u_0$ be the set of all $a \in \overline{a} \cap U_\beta$ for $\beta_* \leq \beta < \delta$ such that $a R e$, and similarly define $u_1, v_0, v_1$. Let $X$ be the set of all $e' \in M$ such that $e' \Delta d = \beta_*$ and $b R e'$ for all $b \in u_0$ and $\lnot b R e'$ for all $b \in u_1$ and $e' R c$ for all $c \in v_0$ and $\lnot e' R c$ for all $c \in v_1$. By condition 2, $X$ contains infinitely many $E_{\beta_* \delta}$-inequivalent elements, so we can choose some $e' \in X$ which is not $E_{\beta_* \delta}$-equivalent to anything in $\overline{a}$. Then $\overline{a} e \mapsto \overline{a} e'$ is an isomorphic embedding, as desired. 
\end{proof}

\begin{lemma}\label{Satisfiable0}
	Suppose $\alpha_*\geq 1$. Then $\Psi_{\alpha_*}$ is satisfiable. In fact, whenever $V$ is a transitive model of $ZFC^-$ with $\alpha_* \in V$, then $\Psi_{\alpha_*}$ has a model in $V$.
\end{lemma}
\begin{proof}
In $V$ construct an ascending sequence $(M_n: n < \omega)$ of models of $\Psi_{\alpha_*}^{\forall}$, such that at each stage we add witnesses to the two conditions of the preceding lemma.
\end{proof}

Moreover:

\begin{lemma}\label{PsiCat}
	Suppose $M, N \models \Psi_{\alpha_*}$. Then $M \equiv_{\infty \omega} N$; in fact, the set of all finite partial isomorphisms from $M$ to $N$ is a back-and-forth system. In particular, for $\alpha_* < \omega_1$, $\Psi_{\alpha_*}$ is $\aleph_0$-categorical.
\end{lemma}
\begin{proof}
    By the definition of $\Psi_{\alpha_*}$.  
\end{proof}

We now connect this with powerset sequences from the preceding section.

\begin{definition}
    Suppose $\overline{X}$ is an $\alpha_*$-powerset sequence. We construe it as an $\mathcal{L}(\Psi_{\alpha_*})$-structure as follows:

    \begin{enumerate}
        \item Each $U_\alpha = X_\alpha$;
        \item For each $\beta$ with $\beta+1 < \alpha_*$, and for all $b \in X_\beta$, $a \in X_\alpha$, we have $b R a$ if and only if $b \in a$;
        \item For each limit $\delta< \alpha_*$ and for each $\beta < \delta$, for all $b \in X_\beta$ and $a\in X_\delta$, we have $b R a$ if and only if $b \in a(\beta)$;
        \item For each $\beta$ with $\beta+1 < \alpha_*$, $E_{\beta \beta+1}$ is discrete;
        \item For each limit $\delta < \alpha_*$ and for each $\beta < \delta$, we have for all $a, b \in X_\delta$ that $a E_{\beta \delta} b$ if and only if $a(\beta) = b(\beta)$ (which implies $a(\beta') = b(\beta')$ for all $\beta' < \beta$).
    \end{enumerate}

    We tacitly identify $\overline{X}$ with this structure, i.e. we view $\overline{X}$ as an $\mathcal{L}(\Psi_{\alpha_*})$-structure.

    \end{definition}

    \begin{lemma}
        If $\overline{X}$ is an $\alpha_*$-powerset sequence then $\overline{X} \models \Psi_{\alpha_*}^{\forall}$.
    \end{lemma}
    \begin{proof}
        Straightforward.
    \end{proof}

    \begin{lemma}
        Suppose $M \models \Psi_{\alpha_*}$, and all elements of $U_0^M$ have the same foundation rank. Then there is a unique $\alpha_*$-powerset sequence $\overline{X}$ with $X_0 = U_0^M$ such that the identity map on $U_0^M$ extends to an isomorphism from $M$ to $\overline{X}$.
    \end{lemma}
    \begin{proof}
        We build $\overline{X}$ and the isomorphism by induction; there is no choice.
    \end{proof}

    \begin{lemma}
        Suppose $\overline{X}$ is a rich $(\lambda, \alpha_*)$-powerset sequence. Then $\overline{X} \models \Psi_{\alpha_*}$.
    \end{lemma}
    \begin{proof}
        We aim to apply Lemma~\ref{ExtraConditions}. 

        We first check the first condition of Lemma~\ref{ExtraConditions} when $\alpha > 0$; the case $\alpha = 0$ is easier. Say $\alpha = \beta+1$. Suppose $u_0, u_1$ are disjoint finite subsets of $X_\beta$ and $v_0, v_1$ are disjoint finite subsets of $\bigcup\{X_\gamma: (\alpha, \gamma) \in \Omega_{\alpha_*}\}$, such that for all $\gamma \in \Omega_{\alpha_*}$ and for all $c_0, c_1$ with $c_i \in v_i \cap X_{\gamma}$, $c_0$ is not $E_{\alpha \gamma}$ -related to $c_1$. Let $\mathcal{O} \subseteq X_\alpha$ be the open set of all $a \in X_\alpha$ with $b R a$ for all $b \in u_0$ and with $\lnot b R a$ for all $b \in u_1$. Since $X_\alpha$ is dense in $\mathcal{P}(X_\beta)$, we get that $\mathcal{O}$ is nonempty. Since $(B_{\alpha \gamma}: (\alpha, \gamma) \in \Omega_{\alpha_*})$ is densely independent over $X_\alpha$, we get the existence of infinitely many  $a \in \mathcal{O}$ such that $a R c$ for all $c \in v_0$ and $\lnot a R c$ for all $c \in v_1$, as desired.

        Now we check the second condition of Lemma~\ref{ExtraConditions} for $\delta$ limit. Suppose $\beta_*, d, u_0, u_1, v_0, v_1$ are as there. Let $\beta < \delta$ be large enough so that each $u_0, u_1 \subseteq \bigcup\{X_{\beta'}: \beta' < \beta\}$. By iteratively applying the fourth condition of the definition of richness, we get infinitely many $E_{\beta_*, \delta}$-inequivalent $a \in X_\delta$ with $a E_{\beta_* \delta} d$ and such that $b R a$ for all $b \in u_0$ and $\lnot b R a$ for all $b \in u_1$. For each such $a$, the $E_{\beta \delta}$-class of $a$ is a nonempty open set, so by dense independence we can find $a' E_{\beta \delta} a$ with $a' R c$ for all $c \in v_0$ and $\lnot a' R c$ for all $c \in v_1$, as desired.
    \end{proof}

\section{Constructing Thick Sets}\label{ThickCompSection}

We aim to prove Theorem~\ref{FatModelsExist1}. Fix throughout this section the following data: $\lambda$, a regular strong limit;  $\alpha_* < \lambda^+$, and $f_* \in \mathbb{F}$. We want to find some $f_*$-closed $A \in \mathbb{V}_{\lambda^+}$ such that $|\mathcal{P}^{\alpha_*}(\lambda) \cap A| = \beth_{\alpha_*}(\lambda)$.  Choose $a_* \in \HC$ containing parameters for $f_*$. Fix some transitive $V \preceq H(\lambda^+)$ of cardinality $\lambda$ with $a_*, \alpha_*, [\lambda]^{<\lambda} \in V$.

We can suppose that $f_*$ satisfies the conclusion of Lemma~\ref{PsiCat2}.  Thus it is enough to construct a rich $(\lambda, \alpha_*)$-sequence $\overline{X}$ and some $f_*$-closed $A \in \mathbb{V}_{\lambda^+}$ with $\overline{X} \in A$. We aim to apply Lemma~\ref{RichExists}; so we must first define $X_0$.

 If $S$ is a set and $\lambda$ is a cardinal, then let $P_{S2\lambda}$ be the set of all partial functions from $S$ to $2$ of cardinality less than $\lambda$. We view $P_{S2\lambda}$ as adding a $\lambda$-Cohen $a \subseteq \mathcal{P}(S)$ (identifying $2^S \cong \mathcal{P}(S)$). Note that if $V$ is a transitive model of $ZFC^-$ with $S\in V$, then $P_{S2\lambda} \in V$, so it makes sense to say when $a \subseteq S$ is $\lambda$-Cohen over $V$. We also view each $(P_{S2 \lambda})^n = P_{S \times n \, 2 \lambda}$, so given $\overline{a} \in (\mathcal{P}(\lambda))^n$, it makes sense to say when $\overline{a}$ is $\lambda$-Cohen over $V$.

\begin{definition}Say $X_0 \subseteq \mathcal{P}(\lambda)$ is $V$-{\em symmetric} if: $X_0$ is $2^\lambda$-dense in $\mathcal{P}(\lambda)$ (with the $<\lambda$-support product topology), and for each injective finite sequence $\overline{a} \in X_0^n$, $\overline{a}$ is $\lambda$-Cohen over $V$. 
\end{definition}

Clearly, if $X_0$ is $V$-symmetric then all elements of $X_0$ have the same foundation rank, namely $\lambda$.

We want to construct some $V$-symmetric $X_0$. The following partial step towards this is a generalization of Claim 6.28 from \cite{CanRamsey} to the case where $\lambda$ is uncountable.

\begin{lemma}
    There is some continuous $h: 2^\lambda \to 2^\lambda$ such that for all injective $\overline{a} \in 2^\lambda$, $h(\overline{a})$ is $\lambda$-Cohen over $V$.
\end{lemma}
\begin{proof}
Let $D_\alpha: \alpha < \lambda$ enumerate all dense open subsets of the various $P_{\lambda 2 \lambda}^n$ for $n < \omega$ which belong to $V$, such that each dense open set occurs cofinally often. We construct, by induction on $\alpha < \lambda$, functions $h_\alpha: 2^\alpha \to 2^{<\lambda}$ so that:

\begin{itemize}
    \item For $\alpha < \beta$ and $s \in 2^\alpha, t \in 2^\beta$ with $s \leq t$, we have $h_\alpha(s) \leq h_\alpha(t)$;
    \item If $D_\alpha$ is dense open in $P_{\lambda 2 \lambda}^n$ then for all $(t_i: i < n)$ distinct from $2^\alpha$, $(h_\alpha(t_i): i < n) \in D_\alpha$. 
\end{itemize}
This is straightforward, using that $\lambda$ is a regular strong limit. Having done so, set $h(f) = \bigcup_{\alpha < \lambda} h_\alpha(f \restriction_\alpha)$ for $f \in 2^\lambda$. Each $h(f)$ is $\lambda$-Cohen over $V$, in particular is a total function; $h$ is easily seen to work.
\end{proof}

\begin{lemma}\label{SymSystemsExist} There exists some $V$-symmetric $X_0 \subseteq \mathcal{P}(\lambda)$. 
\end{lemma}
\begin{proof}
	By the preceding lemma (and using a bijection between $\lambda$ and $\lambda \times \lambda$) we can find $h: 2^\lambda \to 2^{\lambda \times \lambda}$ continuous, so that for all $\overline{a} \in (2^\lambda)^{<\omega}$ injective, $h(\overline{a})$ is $\lambda$-Cohen over $V$. For each $f \in 2^{\lambda \times \lambda}$ and $\gamma < \lambda$ let $f_\gamma \in 2^\lambda$ be defined as $f_\gamma(\nu) = f(\gamma, \nu)$. Let $X_0 := \{h(a)_\gamma: a \in 2^\lambda, \gamma < \lambda\}$. We claim that $X_0$ works. (We are identifying $2^\lambda$ with $\mathcal{P}(\lambda)$.)

    $X_0$ is $2^\lambda$-dense in $2^\lambda$ since for all $a \in 2^\lambda$, $\{h(a)_\gamma: \gamma< \lambda\}$ is dense in $2^\lambda$. (This is because $h(a)$ is $\lambda$-Cohen over $V$.)

    Suppose $(a_i: i < i_*)$ is an injective finite sequence from $2^\lambda$ and $(\gamma_j: j < j_*)$ is a finite injective sequence from $\lambda$. Then $(h(a_i)_{\gamma_j}:i < i_*, j < j_*)$ is $\lambda$-Cohen over $V$, since $(h(a_i): i < i_*)$ is.
\end{proof}
By Lemma~\ref{RichExists}, we can find some rich $(\lambda, \alpha_*)$-powerset sequence $\overline{X}$ with $X_0$ being $V$-symmetric over $\mathcal{P}(\lambda)$.

If we can find some $f_*$-closed $A \in \mathbb{V}_{\lambda^+}$ with $\overline{X} \in A$, then we are done, by choice of $f_*$. So we aim to find some such $A$.

By Lemma~\ref{Satisfiable0}, we can find some $N \models \Psi_{\alpha_*}$ with $N \in V$. Let $P$ be the set of all finite partial isomorphisms from $N$ to $\overline{X}$. By Lemma~\ref{PsiCat}, $P$ adds an isomorphism $\dot{\mathbf{\sigma}}: \check{N} \cong \check{\overline{X}}$.

We identify $2^\lambda$ with $\mathcal{P}(\lambda)$, and so view $X_0 \subseteq 2^\lambda$. Let $Q = \prod_{U_0^N} P_{\lambda 2 \lambda}$ with finite supports, so $Q \in V$. Rephrased, $Q$ is the set of all finite partial functions $f$ from $U_0^N$ to $P_{\lambda 2 \lambda}$. 

Let $\dot{g}$ be the $P$-name for $\dot{\sigma} \restriction_{\check{U}_0^{\check{N}}}: \check{U}_0^{\check{N}} \to 2^\lambda$.

\begin{lemma}\label{BasicCount1}
	$P$ forces that $\dot{g}$ is $\check{Q}$-generic over $\check{V}$.
\end{lemma}
\begin{proof}
	Suppose $D$ is a dense subset of $Q$ in $V$ and $\sigma: N \to \overline{X}$ is a finite partial isomorphism. It suffices to show that we can find some $\tau$ extending $\sigma$, so that $\tau \restriction_{U_0^N}$ extends an element of $D$.
	
	Let $u = \mbox{dom}(\sigma) \cap U_0^N$, a finite subset of $U_0^N$. By choice of $X_0$, we have that $\sigma \restriction_u$ is $P_{\lambda 2 \lambda}^u$-generic over $V$. We have the natural restriction map $\pi: Q \to P_{\lambda 2 \lambda}^u$. $\pi[D]$ is dense in $P_{\lambda 2 \lambda}^u$ and is in $V$, so we can find $t_0 \in \pi[D]$ with $t_0 \subseteq \sigma \restriction_u$. Choose $t \in D$ with $\pi(t) = t_0$. It suffices to show we can find some $\tau \in P$ extending $\sigma$, such that $\tau \restriction_{U_0^N}$ extends $t$.
	
	Enumerate $\mbox{dom}(t) \backslash u = \{a_i: i < n\}$. For each $i < n$, let $\mathcal{O}_i$ be the basic open subset of $2^\lambda$ determined by $t(a_i)$, namely $\mathcal{O}_i$ is the set of extensions of $t(a_i)$ to $2^\lambda$. By extending $t$, we can suppose $(\mathcal{O}_i: i < n)$ are pairwise disjoint, and that for each $i <n$ and for each $a \in u$, $\sigma(a) \not \in \mathcal{O}_i$.
	
    Let $v = \{a \in \mbox{dom}(\sigma): a \in U_\alpha^N \mbox{ for some } \alpha \mbox{ with } (0, \alpha) \in \Omega_{\alpha_*}\}$, a finite subset of $N$. For each $i < n$, by dense independence, we can find some $\nu_i \in \mathcal{O}_i$ such that for each $b \in v$, $\nu_i R^{\overline{X}} \sigma(b)$ if and only if $a_i R^N b$. Define $\tau(a) = \sigma(a)$ for all $a \in \mbox{dom}(\sigma)$, and define $\tau(a_i) = \nu_i$ for each $i < n$. Then $\tau \in P$ extends $\sigma$, and further $\tau \restriction_{U_0^N}$ extends $t$.
\end{proof}

Now we finish. Working in $\mathbb{V}$, let $A$ be the least $f_*$-closed set with $\overline{X} \in A$ (that is, the least transitive set $A$ which is closed under $(f_*)_{\ptl}$ and closed under finite sequences). We need to show that $A \in \mathbb{V}_{\lambda^+}$.

Let $\mathbb{V}[G]$ be a $P$-generic forcing extension of $\mathbb{V}$, and let $\sigma = \mbox{val}(\dot{\sigma}, G)$, the generic isomorphism from $N$ to $\overline{X}$. Write $g =  \sigma \restriction_{U_0^N}$; by Lemma~\ref{BasicCount1}, $g$ is $Q$-generic over $V$, hence $V[g]$ is a forcing extension of $V$. But then $\overline{X} \in V[g]$: it can be recovered as the unique $\alpha_*$-powerset sequence with $X_0 = g[U_0^N]$, such that $g$ extends to an isomorphism from $N$ to $\overline{X}$. Now $V[g]$ is closed under $(f_{*})_{\ptl}$ by Lemma~\ref{ptlAbsolute} and $\overline{X} \in V[g]$, so $A \subseteq V[g]$, so $\mbox{rnk}(A) \leq \mbox{rnk}(V[g]) = \mbox{rnk}(V)$.

Back in $\mathbb{V}$, this means $A \in \mathbb{V}_{\lambda^+}$.

\section{Schr\"{o}der--Bernstein Properties}\label{BCSBSBSec}

In this section, we define various Schr\"{o}der--Bernstein properties of sentences $\Phi \in \mathcal{L}_{\omega_1 \omega}$. In the next section, we apply the thickness machinery to show that these properties imply a bound on the complexity of countable models of $\Phi$, assuming large cardinals. 

There is significant freedom in what definition of embedding to use. We could just use elementary embeddings, but that would not be general enough for some applications in future work. Instead we will look at embeddings that preserve $\equiv_{\alpha \omega}$-type, as defined below. The case of elementary embeddings can be recovered with $\alpha = 0$.

\begin{definition}
    Given two structures $M, N$ and $\overline{a} \in M, \overline{b} \in N$, we define what it means for $(M, \overline{a}) \equiv_{\alpha \omega} (N, \overline{b})$ by induction on $\alpha$.

    Say that $(M, \overline{a}) \equiv_{0 \omega} (N, \overline{b})$ if $\overline{a}$ and $\overline{b}$ have the same first-order type. Say that $(M, \overline{a}) \equiv_{\alpha +1, \omega} (N, \overline{b})$ if for every $a \in M$ there is $b \in N$ with $(M, \overline{a}a) \equiv_{\alpha \omega} (N, \overline{b} b)$, and vice versa. Say that $(M, \overline{a}) \equiv_{\delta \omega} (N, \overline{b})$ if for all $\alpha < \delta$, $(M, \overline{a}) \equiv_{\alpha \omega} (N, \overline{b})$.
\end{definition}
\begin{definition}
	Suppose $\mathcal{L}$ is a language, and $M, N$ are $\mathcal{L}$-structures, and $\alpha$ is a countable ordinal. Then say that $f: M \preceq_{\alpha \omega} N$ is an $\equiv_{\alpha \omega}$-{\em embedding} if for all $\overline{a} \in M$, $(M, \overline{a}) \equiv_{\alpha \omega} (M, f(\overline{a})$. Say that $M \sim_{\alpha \omega} N$ if $M \preceq_{\alpha \omega} N \preceq_{\alpha \omega} M$.

 Given an ordinal $\alpha$ and potential canonical Scott sentences $\phi, \psi$, say that $\phi \preceq_{\alpha\omega} \psi$ if for some or any forcing extension in which $\phi, \psi, \alpha$ are hereditarily countable, and for some or any countable $M \models \phi, N \models \psi$, we have $M \preceq_{\alpha\omega} N$; and similarly define $\sim_{\alpha \omega}$ on potential canonical Scott sentences. 
\end{definition}

\begin{definition}
For $\alpha$ a countable ordinal, say that $(\Phi, \sim_{\alpha \omega})$ has the {\em Schr\"{o}der--Bernstein property} if for all $M, N \models \Psi$ countable, if $M \sim_{\alpha \omega} N$ then $M \cong N$. For an arbitrary ordinal $\alpha$, say that $(\Phi, \sim_{\alpha \omega})$ has the {\em potential Schr\"{o}der--Bernstein property} if for for all potential canonical Scott sentences $\phi, \psi$ of $\Phi$, if $\phi \sim_{\alpha \omega} \psi$ then $\phi = \psi$. 
\end{definition}

We aim to show that whenever $(\Phi, \sim_{\alpha \omega})$ has the (potential) Schr\"{o}der--Bernstein property, then under appropriate large cardinals this implies a bound on thickness.  First off, we show there is no conflict if we drop the adjective ``potential:"

\begin{lemma}
Supose $\alpha$ is a countable ordinal and $\Phi$ is a sentence of $\mathcal{L}_{\omega_1 \omega}$. Then $(\Phi, \sim_{\alpha \omega})$ has the Schr\"{o}der--Bernstein property if and only if it has the potential Schr\"{o}der--Bernstein property.
\end{lemma}
\begin{proof}
    The reverse direction is immediate. For the forward direction, use Shoenfield Absoluteness to get that $(\Phi, \sim_{\alpha \omega})$ retains the Schr\"{o}der--Bernstein property in every forcing extension, from which the potential Schr\"{o}der--Bernstein property follows.
\end{proof}
We can thus drop the adjective ``potential."

We want a more explicit characterization of what it means for a pair of potential canonical Scott sentences $\phi, \psi$ to have $\phi \sim \psi$. For this, it is helpful to consider colored trees.

\begin{definition}
	 Suppose $C$ is a set. A {\em $C$-colored tree} is a structure $\mathcal{T} = (T,\leq_{\mathcal{T}}, 0_{\mathcal{T}}, c_{\mathcal{T}})$ where $(T, \leq_{\mathcal{T}}, 0_{\mathcal{T}})$ is a tree of height at most $\omega$ with root $0_{\mathcal{T}}$, and $c_{\mathcal{T}}: T \to C$ is a coloring (formally given by a sequence of disjoint unary predicates indexed by some $C' \subseteq C$, namely the used colors). A {\em colored tree} is a $C$-colored tree for some $C$. If $C \subseteq \HC$ is absolutely $\Delta_1$ then let $\mbox{CT}_{\aleph_0}(C)$ denote all hereditarily countable $C$-colored trees. Given any class $C$ let $\mbox{CT}_{\infty}(C)$ be the class of all $C$-colored trees. So $\mbox{CT}_{\aleph_0}(X)_{\ptl} =\mbox{CT}_{\infty}(X_{\ptl})$.

  Say that $f: \mathcal{T} \leq^{\mbox{\small{ct}}} \mathcal{T}' $ is an embedding of colored trees if $f(0_{\mathcal{T}}) = 0_{\mathcal{T}'}$, and $f$ preserves height and colors, and for all $s, t \in \mathcal{T}$, $s \leq_{\mathcal{T}} t$ implies $f(s) \leq_{{\mathcal{T}}'} f(t)$. $f$ need not be injective. Say that $\mathcal{T} \sim^{\mbox{\small{ct}}} \mathcal{T}'$ if $\mathcal{T} \leq^{\mbox{\small{ct}}} \mathcal{T}' \leq^{\mbox{\small{ct}}} \mathcal{T}$.
	
	If $\mathcal{T}$ is a colored tree and $s \in \mathcal{T}$, let $\mathcal{T}_{\geq s}$ be the colored tree with root $0_{\mathcal{T}_{\geq s}} = s$, consisting of all elements of $\mathcal{T}$ extending $s$.
\end{definition}

\begin{definition}
	Suppose $\mathcal{T}, \mathcal{T}'$ are colored trees. We define what it means for $\mathcal{T} \leq_\alpha^{\mbox{\small{ct}}} \mathcal{T}'$ by induction on $\alpha$. Put $\mathcal{T} \leq_0^{\mbox{\small{ct}}} \mathcal{T}'$ if $c_{\mathcal{T}}(0_{\mathcal{T}}) = c_{\mathcal{T}'}(0_{\mathcal{T}'})$. For $\delta$ limit, put $\mathcal{T} \leq_\delta^{\mbox{\small{ct}}}\mathcal{T}'$ if $\mathcal{T} \leq_\alpha^{\mbox{\small{ct}}}\mathcal{T}'$ for all $\alpha < \delta$. Finally, put $\mathcal{T} \leq_{\alpha+1}^{\mbox{\small{ct}}} \mathcal{T}'$ if $0_{\mathcal{T}}$ and $0_{\mathcal{T}'}$ have the same color, and for all $s \in \mathcal{T}$ an immediate successor of  $0_{\mathcal{T}}$, there is $s' \in \mathcal{T}'$ an immediate successor of $0_{\mathcal{T}'}$, such that $\mathcal{T}_{\geq s} \sim_{\alpha}^{\mbox{\small{ct}}} \mathcal{T}'_{\geq s'}$. 
\end{definition}

\begin{lemma}\label{EmbeddingsFiltrate}
Suppose $\mathcal{T}, \mathcal{T}'$ are colored trees. Then 
    $\mathcal{T} \leq^{\mbox{\small{ct}}} \mathcal{T}'$ if and only if $\mathcal{T} \leq_\alpha^{\mbox{\small{ct}}} \mathcal{T}'$ for all $\alpha$. In fact, this is a theorem of $ZFC^-$.
\end{lemma}

The following lemma is very special to colored trees. For instance, the embeddability relation on uncountable dense linear orders is very complicated, even though DLO is $\aleph_0$-categorical.

\begin{lemma}\label{EmbeddingIsAbsoluteOnCT}
	Suppose $V$ is a transitive model of $ZFC^-$, and $\mathcal{T}, \mathcal{T}' \in V$ are colored trees. Then for $R \in \{\leq^{\mbox{\small{ct}}}, \sim, \leq^{\mbox{\small{ct}}}_\alpha: \alpha \in V\}$, we have that $\mathcal{T} \,\, R \,\,  \mathcal{T}'$ if and only if $(\mathcal{T} \,\, R \,\, \mathcal{T}')^V$.
\end{lemma}
\begin{proof}
	First, an easy induction on $\alpha \in V$ shows that for all $\mathcal{T}, \mathcal{T}' \in V$, $\mathcal{T} \leq^{\mbox{\small{ct}}}_\alpha \mathcal{T}'$ if and only if $(\mathcal{T} \leq^{\mbox{\small{ct}}}_\alpha \mathcal{T}')^V$. 
	
	Now, if $\mathcal{T} \leq^{\mbox{\small{ct}}} \mathcal{T}'$, then in particular $\mathcal{T} \leq^{\mbox{\small{ct}}}_\alpha \mathcal{T}'$ for all $\alpha \in V$, so $(\mathcal{T} \leq^{\mbox{\small{ct}}} \mathcal{T}')^V$. For the converse, suppose $(\mathcal{T} \leq^{\mbox{\small{ct}}} \mathcal{T}')^V$; choose $f \in V$ with $f: \mathcal{T} \leq^{\mbox{\small{ct}}} \mathcal{T}'$. Then $f$ witnesses $\mathcal{T} \leq^{\mbox{\small{ct}}} \mathcal{T}'$.
\end{proof}

We also need the following notation for Scott sentences, following \cite{URL}:

\begin{definition}\label{ScottSentenceNotation}
	If $\phi$ is a canonical Scott sentence -- that is, $\phi\in \CSS(\mathcal{L})_\ptl$ -- then let $S^n_\infty(\phi)$ be the set of all potential canonical Scott sentences in the language $\mathcal{L}'=\mathcal{L}\cup \{c_0,\ldots,c_{n-1}\}$ which imply $\phi$.  We will refer to elements of $S^n_\infty(\phi)$ as types, i.e. as infinitary formulas with free variables $x_0,\ldots,x_{n-1}$, resulting from replacing each $c_i$ with a new variable $x_i$ not otherwise appearing in the formula.  It is equivalent to define $S^n_\infty(\phi)$ by forcing: if $\mathbb{V}[H]$ makes $\phi$ hereditarily countable and $M\in \mathbb{V}[H]$ is the unique countable model of $\phi$, then $S^n_\infty(\phi)$ is the set $\{\css(M,\overline{a}):\overline{a}\in M^n\}$. 
\end{definition}

The following lemma, combined with Lemma~\ref{EmbeddingIsAbsoluteOnCT}, explains why embedding on colored trees is useful for us. 

\begin{lemma}\label{EmbeddingAbsolute}
	Suppose $\Phi \in \mathcal{L}_{\omega_1 \omega}$. Then there is an absolutely $\Delta_1$ map $f: \mbox{CSS}(\Phi) \times \omega_1  \to \mbox{CT}_{\aleph_0}(\HC)$, such that persistently, the following holds: for all $\phi, \psi \in \CSS(\Phi)$, and for all countable ordinals $\alpha$, $\phi \sim_{\alpha \omega} \psi$ if and only if $f(\phi, \alpha) \sim^{\mbox{\small{ct}}} f(\psi, \alpha)$, and further, for all $\phi \in \mbox{CSS}(\Phi)$ and $\alpha < \omega_1$, $f(\phi, \alpha) \in \mbox{CT}_{\aleph_0}(\HC_{\omega + \alpha+1})$.
\end{lemma}
\begin{proof}
	
	Suppose $\phi \in \mbox{CSS}(\Phi)$ and $\alpha < \omega_1$; we describe how to construct $f(\phi, \alpha)$. For each $n < \omega$, let $S^n_\infty(\phi) $ be as defined in Definition~\ref{ScottSentenceNotation}. Then $S^{<\omega}_{\infty}(\phi) = \bigcup_n S^n_\infty(\phi)$ naturally forms a tree whose $n$'th level is $S^n_\infty(\phi)$. Given an ordinal $\alpha$, we color each node $p(\overline{x}) \in S^n_\infty(\phi)$ by an element of $\HC_{\omega+\alpha+1}$ encoding the type of $(M, \overline{a})$ up to $\equiv_{\alpha \omega}$-equivalence, for some or any $(M, \overline{a}) \models p(\overline{x})$. Let $f(\phi, \alpha)$ denote the colored tree so obtained. We claim that $f$ works. It is clearly $a \Delta_1$.

    After passing to a forcing extension, it suffices to show that for all $\phi, \psi \in \CSS(\Phi)$ and for all countable ordinals $\alpha$, we have that $\phi \preceq_{\alpha \omega} \psi$ if and only if $f(\phi, \alpha) \leq^{\mbox{\small{ct}}} f(\psi, \alpha)$. Let $\phi, \psi, \alpha$ be given and let $M \models \phi, N \models \psi$ be countable.

    A straightforward induction allows us to choose $\overline{a}_p: p(\overline{x}) \in S^{<\omega}_\infty(\phi)$ with $\overline{a}_p \in M^{|\overline{x}|}$, such that each $(M, \overline{a}_p) \models p(\overline{x})$ and such that $p(\overline{x}) \leq q(\overline{x} \overline{y})$ (in the tree order on $S^{<\omega}_\infty(\phi)$) implies $\overline{a}_p \subseteq \overline{a}_q$. Similarly choose $\overline{b}_p: p(\overline{x}) \in S^{<\omega}_\infty(\psi)$ from $N^{<\omega}$.
    
    Suppose now that there is some $h: M \preceq_{\alpha\omega} N$. Then the map $S^{<\omega}_\infty(\phi) \to S^{<\omega}_\infty(\psi)$ sending $p(\overline{x})$ to $\css(N, h(\overline{a}_p))$ is an embedding of colored trees.

    Finally, suppose that there is some tree embedding from $\tilde{h}: S^{<\omega}_\infty(\phi)\to S^{<\omega}_\infty(\psi)$ preserving the colors. Let $a_n: n < \omega$ be an enumeration of $M$. For each $n$ let $p_n(x_i: i < n)$ be equal to $\tilde{h}(\css(M, a_i: i < n))$. Then $n < m$ implies $p_n \leq p_m$, so $\overline{b}_{p_n} \subseteq \overline{b}_{p_m}$, so we can unambiguosly write $\overline{b}_{p_n} = (b_i: i < n)$. Then $(a_n \mapsto b_n: n < \omega)$ is the desired $\preceq_{\alpha \omega}$-embedding.
\end{proof}

\section{Counting Colored Trees up to Bi-embeddability}\label{BCSBCountingSec}

In this section, we show that if $\Phi$ has the Schr\"{o}der--Bernstein property, then $\Phi$ is not Borel complete, assuming a certain large cardinal. Specifically, we will need the Erd{\H o}s cardinals: 

\begin{definition}
	Suppose $\alpha$ is an ordinal (we will only use the case $\alpha = \omega$). Then let $\kappa(\alpha)$ be the least cardinal $\kappa$ with $\kappa \rightarrow (\alpha)^{<\omega}_2$ (if it exists). In words: whenever $F: [\kappa(\alpha)]^{<\omega} \to 2$, there is some $X \subseteq \kappa(\alpha)$ of ordertype $\alpha$, such that $F \restriction_{[X]^n}$ is constant for each $n < \omega$.
	
\end{definition}

$\kappa(\omega)$ is a large cardinal: it is always inaccessible and has the tree property. On the other hand, it is absolute to $\mathbb{V} = \mathbb{L}$, and well below the consistency strength of a measurable cardinal. See \cite{Kanamori} for a description of these results.

The following is a theorem of Shelah \cite{ShelahTrees}; see \cite{TFAGWShelah} for a more direct proof. In the following theorem, the term ``antichain" is used in the sense of well-quasi-ordering theory (rather than in the sense of forcing theory), so $A$ is an antichain of for all $a, b \in A$, $a \not \leq b$ and $b \not \leq a$.

\begin{theorem}\label{ShelahTreesThm}
	Suppose $\kappa(\omega)$ exists. Then $\omega$-colored trees, under $\leq^{\mbox{\small{ct}}}$,  is a $\kappa(\omega)$-well-quasi-order (in fact a $\kappa(\omega)$-better-quasi-order). In other words, it has no descending chains nor antichains of size $\kappa(\omega)$.

    Moreover, for every $\kappa < \kappa(\omega)$, there is an antichain of $\omega$-colored trees of length $\kappa$.
\end{theorem}

This theorem is a fundamental constraint on the complexity of bi-embeddability relations, and will allow us to bound the complexity of sentences with the Schr\"{o}der--Bernstein property. First we note a simple generalization, see also \cite{ShelahTrees}.

\begin{corollary}
Suppose $\kappa(\omega) < \infty$ and $C$ is a set of size less than $\kappa(\omega)$. Then $C$-colored trees, under $\leq^{\mbox{\small{ct}}}$, is a $\kappa(\omega)$-well-quasi-order. 
\end{corollary}
\begin{proof}
We can suppose $C$ is a cardinal $\kappa < \kappa(\omega)$. Let $(\mathcal{T}_\alpha: \alpha < \kappa)$ be an antichain of $\omega$-colored trees, which we suppose do not use the color $0$. Then given any $\kappa$-colored tree $\mathcal{T}$, let $\mathcal{S}$ be the $\omega$-colored-tree obtained from $\mathcal{T}$ by putting a copy of $\mathcal{T}_\alpha$ above every element of $\mathcal{T}$ of color $\alpha$, and then recoloring the original tree $\mathcal{T}$ by $0$. This provides a $\leq^{\mbox{\small{ct}}}$-preserving embedding from $\kappa$-colored trees into $\omega$-colored trees, which proves the corollary.
\end{proof}
Before proceeding, we want the following definition and technical lemma from \cite{ShelahTrees} (see the proof of Theorem 5.3 there). These allow us to replace general colored trees by well-founded colored trees.

\begin{definition}
	Suppose $\mathcal{T}$ is a colored tree and $\alpha$ is an ordinal. Then let $\mathcal{T} \times \alpha$ denote the colored tree of all pairs $(s, \overline{\beta})$, where $s \in \mathcal{T}$ is of height $n$, and $\overline{\beta} = (\beta_0, \ldots, \beta_{n-1})$ is a strictly decreasing sequence of ordinals with $\beta_0 < \alpha$. We define $c_{\mathcal{T} \times \alpha}(s, \overline{\beta}) = c_{\mathcal{T}}(s)$.
\end{definition}

\begin{lemma}\label{RedToWF}
	Suppose $\mathcal{T}, \mathcal{T}'$ are colored trees. Then for all ordinals $\alpha$,  $\mathcal{T} \times \alpha \leq^{\mbox{\small{ct}}}_\alpha \mathcal{T}' \times \alpha$ if and only if $\mathcal{T} \leq_\alpha^{\mbox{\small{ct}}} \mathcal{T}'$ if and only if $\mathcal{T} \times \alpha \leq^{\mbox{\small{ct}}} \mathcal{T}' \times \alpha$.
\end{lemma}
\begin{proof}
	We verify by induction on $\alpha$ that $\mathcal{T} \times \alpha \leq^{\mbox{\small{ct}}}_\alpha \mathcal{T}' \times \alpha$ implies $\mathcal{T} \leq_\alpha^{\mbox{\small{ct}}} \mathcal{T}'$ implies $\mathcal{T} \times \alpha \leq^{\mbox{\small{ct}}} \mathcal{T}' \times \alpha$ (the remaining implication follows from Lemma~\ref{EmbeddingsFiltrate}). $\alpha = 0$ is immediate.
	
	Successor stage first implication: suppose $\mathcal{T} \times (\alpha+1) \leq^{\mbox{\small{ct}}}_{\alpha+1} \mathcal{T}' \times (\alpha+1)$, and let $s \in \mathcal{T}$ be an immediate successor of $0_{\mathcal{T}}$. Let $(s', \beta) \in \mathcal{T}' \times (\alpha+1)$ be an immediate successor of $0_{\mathcal{T}' \times (\alpha+1)}$ such that $(\mathcal{T} \times (\alpha+1))_{\geq (s, \alpha)} \leq^{\mbox{\small{ct}}}_\alpha (\mathcal{T}' \times (\alpha+1))_{ \geq (s', \beta)}$. This means that $(\mathcal{T}_{\geq s}) \times \alpha \leq^{\mbox{\small{ct}}}_\alpha (\mathcal{T}'_{\geq s'}) \times \beta$ (since the corresponding trees are isomorphic). But easily $(\mathcal{T}'_{\geq s'}) \times \beta \leq^{\mbox{\small{ct}}} (\mathcal{T}'_{\geq s'}) \times \alpha$, so we get that $(\mathcal{T}_{\geq s}) \times \alpha \leq^{\mbox{\small{ct}}}_\alpha (\mathcal{T}'_{\geq s'}) \times \alpha$. Thus, by the inductive hypothesis $\mathcal{T}_{\geq s} \leq^{\mbox{\small{ct}}}_\alpha \mathcal{T}'_{\geq s'}$. 
	
	Successor stage, second implication: suppose $\mathcal{T} \leq_{\alpha+1}^{\mbox{\small{ct}}} \mathcal{T}'$; given $(s, (\beta)) \in \mathcal{T} \times (\alpha+1)$ an immediate successor of $0_{\mathcal{T} \times (\alpha+1)}$, choose $s' \in \mathcal{T}'$ an immediate successor of $0_{\mathcal{T}}$ such that $\mathcal{T}_{\geq s}\leq_\beta^{\mbox{\small{ct}}} \mathcal{T}'_{\geq s'}$, and note by the inductive hypothesis that $(\mathcal{T} \times (\alpha+1))_{\geq (s, (\beta))} \cong \mathcal{T}_{\geq s} \times \beta \leq^{\mbox{\small{ct}}} \mathcal{T}'_{\geq s'} \times \beta \cong (\mathcal{T}' \times (\alpha+1))_{\geq (s', (\beta))}$. Thus there is an embedding from $(\mathcal{T} \times (\alpha+1))_{\geq(s, (\beta))}$ to $(\mathcal{T}' \times (\alpha+1))_{\geq (s', (\beta))}$. Patching together such embeddings for all $(s, (\beta))$ gives the desired embedding of $\mathcal{T} \times (\alpha+1)$ into $\mathcal{T}' \times (\alpha+1)$.
	
	Limit stage, first implication: suppose $\mathcal{T} \times \delta \leq^{\mbox{\small{ct}}}_\delta \mathcal{T}' \times\delta$. Thus, for all $\alpha < \delta$, $\mathcal{T} \times \delta \leq^{\mbox{\small{ct}}}_\alpha \mathcal{T}' \times\delta$. By the inductive hypothesis, this implies that for all $\alpha < \delta$, $(\mathcal{T} \times \delta) \times \alpha) \leq^{\mbox{\small{ct}}}_\alpha (\mathcal{T}' \times \delta)\times \alpha$, but always $(\mathcal{S} \times \gamma_0) \times \gamma_1 \sim^{\mbox{\small{ct}}} \mathcal{S} \times \mbox{min}(\gamma_0, \gamma_1)$, so we get that $(\mathcal{T} \times \alpha) \leq^{\mbox{\small{ct}}}_\alpha (\mathcal{T}' \times \alpha)$, hence by the inductive hypothesis again $\mathcal{T} \leq^{\mbox{\small{ct}}}_\alpha \mathcal{T}'$. This holds for all $\alpha < \delta$ so $\mathcal{T} \leq^{\mbox{\small{ct}}}_\delta \mathcal{T}'$.
	
	Limit stage, second implication: suppose $\mathcal{T} \leq^{\mbox{\small{ct}}}_\delta \mathcal{T}'$. Then by definition of $\leq^{\mbox{\small{ct}}}_\delta$ and the inductive hypothesis, we get that $\mathcal{T} \times \alpha \leq^{\mbox{\small{ct}}} \mathcal{T}' \times \alpha$ for all $\alpha < \delta$. Since $\mathcal{T}' \times \alpha \leq^{\mbox{\small{ct}}} \mathcal{T} \times \delta$ and $\mathcal{T} \times \delta = \bigcup_{\alpha < \delta} \mathcal{T} \times \alpha$, we get that $\mathcal{T} \times \delta \leq^{\mbox{\small{ct}}} \mathcal{T}' \times \delta$.
	
\end{proof}

We can now prove the following. To fix notation, if $\mathcal{T}$ is a well-founded colored tree and $t \in T$, then inductively define $\mbox{rnk}(\mathcal{T}, t) = \mbox{sup}\{\mbox{rnk}(\mathcal{T}, s)+1: s $ an immediate successor of $t\}$. Define $\mbox{rnk}(\mathcal{T}) = \mbox{rnk}(\mathcal{T}, 0_{\mathcal{T}})$.
\begin{lemma}\label{TreeCountLemma}
	Suppose $\kappa(\omega)$ exists and $C$ is a set of size $< \kappa(\omega)$. Suppose $\alpha$ is a nonzero ordinal. Then there are at most $|\alpha|^{<\kappa(\omega)}$ $C$-colored trees $\mathcal{T}$ with $\mbox{rnk}(\mathcal{T}) < \alpha$, up to bi-embeddability. 
\end{lemma}
\begin{proof}
A routine computation shows that for every ordinal $\alpha < \kappa(\omega)$, there are at most $\beth_\alpha(|C|)$-many $C$-colored trees of rank $<\alpha$ up to isomorphism, and hence up to bi-embeddability. Since $\beth_{\kappa(\omega)}(|C|)= \kappa(\omega)$, the lemma is true for all $\alpha \leq \kappa(\omega)$. To finish, we proceed by induction on $\alpha \geq \kappa(\omega)$. 
	
	The case $\alpha$ limit is trivial, since $\sum_{\beta < \alpha} |\beta|^{<\kappa(\omega)} \leq |\alpha|^{<\kappa(\omega)}$. 
	
	Suppose we are at stage $\alpha+1$. Write $\kappa = |\alpha|$. Let $\mathbb{S}$ be a choice of representatives for well-founded $C$-colored trees of rank $< \alpha$ up to bi-embeddability; so $|\mathbb{S}| \leq \kappa^{<\kappa(\omega)}$. Suppose $\mathcal{T} = (T, <, 0, c)$ is given of rank $\alpha$. Let $X_{\mathcal{T}}$ be the set of all $\mathcal{S} \in \mathbb{S}$ such that there is some $t \in \mathcal{T}$ of height $1$ such that $\mathcal{S}$ embeds into the colored tree $\mathcal{T}_{\geq t}$. Note that $\mathcal{T}$ is bi-embeddable with the tree $(\mathcal{T}', <', 0', c')$, which is defined by: $c'(0') = c(0)$, and then we put a copy of each $\mathcal{S} \in X_{\mathcal{T}}$ above $0'$. Thus, $\mathcal{T}/ \sim^{\mbox{\small{ct}}}$ is determined by the pair $(X_{\mathcal{T}}, c(0_{\mathcal{T}}))$, where $X_{\mathcal{T}}$ is a downward-closed subset of $\mathbb{S}$ (ordered by embeddability $\leq$).
	
	Thus, it suffices to show there are only $\kappa^{<\kappa(\omega)}$-many downward closed subsets of $\mathbb{S}$.
	
	Suppose $X \subseteq \mathbb{S}$ is downward closed. It is straightforward to find a subtree $\mathbf{T}$ of $\mathbb{S}^{<\kappa(\omega)}$ (of uncountable height) such that:
	
	\begin{itemize}
		\item Whenever $(\mathcal{S}_\beta: \beta < \alpha) \in \mathbf{T}$, then for all $\beta < \beta' < \alpha$, $\mathcal{S}_\beta >^{\mbox{\small{ct}}} \mathcal{S}_{\beta'}$, and each $\mathcal{S}_\beta \not \in X$;
		
		\item For each $\overline{S} = (\mathcal{S}_\beta: \beta < \alpha) \in \mathbf{T}$, the set of all $\mathcal{S} \in \mathbb{S}$ such that $\overline{S} \mathcal{S} \in \mathbf{T}$ forms a maximal antichain in $\{\mathcal{S} \in \mathbb{S} \backslash X: \mathcal{S} <^{\mbox{\small{ct}}} \mathcal{S}_\beta$ for all $\beta < \alpha\}$.
	\end{itemize}
	
	$\mathbf{T}$ is of height at most $\kappa(\omega)$ (being a subtree of $\mathbb{S}^{<\kappa(\omega)}$); since $(\mathbb{S}, \leq)$ has no descending chains of length  $\kappa(\omega)$, $\mathbf{T}$  has no branches of length $\kappa(\omega)$. Further, since $\kappa(\omega)$ is inaccessible and $(\mathbb{S}, \leq)$ has no antichains of size $\kappa(\omega)$, each level of $\mathbf{T}$ must have size less than $\kappa(\omega)$. Thus, since $\kappa(\omega)$ has the tree property, $\mathbf{T}$ must be of height less than $\kappa(\omega)$; thus $|\mathbf{T}| < \kappa(\omega)$. Thus, it suffices to show that $X$ is determined by $\mathbf{T}$, since $|\mathbb{S}|^{<\kappa(\omega)} \leq \kappa^{<\kappa(\omega)}$. Define $Y = \{\mathcal{S} \in \mathbb{S}: \mbox{ there is no } (\mathcal{S}_\beta: \beta \leq \alpha) \in \mathbf{T} \mbox{ with } \mathcal{S}_\alpha \leq^{\mbox{\small{ct}}} \mathcal{S}\}$; it suffices to show that $X = Y$.
	
	It follows immediately from the construction of $\mathbf{T}$ that $X \subseteq Y$; so it suffices to show that $Y \subseteq X$. So suppose $\mathcal{S} \not \in X$; we show $\mathcal{S} \not \in Y$. Define a chain $(\mathcal{S}_\beta: \beta < \beta_*) \in \mathbf{T}$ inductively, so that each $\mathcal{S}_\beta >^{\mbox{\small{ct}}} \mathcal{S}$, for as long as possible. This process must stop before $\kappa(\omega)$, say we cannot find $\mathcal{S}_{\beta_*}$ with $\beta_* < \kappa(\omega)$. Let $\mathcal{A} = \{\mathcal{S}' \in \mathbb{S}: (\mathcal{S}_\beta: \beta < \beta_*)^\frown \mathcal{S}' \in \mathbf{T}\}$. Every element of $\mathcal{A}$ is either incomparable with or below $\mathcal{S}$; by maximality of $\mathcal{A}$, there must be some $\mathcal{S}_{\beta_*} \in \mathcal{A}$ with $\mathcal{S}_{\beta_*} \leq^{\mbox{\small{ct}}} \mathcal{S}$, so $\mathcal{S} \not \in Y$.
\end{proof}

The preceding lemma shows there are few well-founded colored trees in some sense; the following lemma shows that there are few colored trees in some sense.

\begin{lemma}\label{TreeCount0}
    Suppose $\kappa(\omega)$ exists. Then there is some $f \in \mathbb{F}$ such that for all $\alpha < \kappa(\omega)$, for all regular cardinals $\lambda$, and for all $f$-closed $A \in \mathbb{V}_{\lambda^+}$, $A$ contains at most $\lambda^{<\kappa(\omega)}$-many $\mathbb{V}_\alpha$-colored trees up to bi-embeddability.
\end{lemma}
\begin{proof}
    Define $f: \mbox{CT}_{\aleph_0}(\HC) \times \mbox{CT}_{\aleph_0}(\HC) \to \omega_1$ via $f(\mathcal{T}, \mathcal{T}') = 0$ if $\mathcal{T} \leq^{\mbox{\small{ct}}} \mathcal{T}'$, and otherwise $f(\mathcal{T}, \mathcal{T}') = $ the least $\alpha$ such that $\mathcal{T} \not \leq_\alpha^{\mbox{\small{ct}}} \mathcal{T}'$. By Theorem~\ref{EmbeddingIsAbsoluteOnCT}, $f$ is absolute to transitive models of $ZFC^-$, hence is absolutely $\Delta_1$. We claim that $f$ works.

    Indeed, suppose $A \in \mathbb{V}_{\lambda^+}$ is $f$-closed, and let $\alpha_* = A \cap \mbox{ON}$, so $\alpha_* < \lambda^+$. 
	
	We claim that $A$ contains at most $\lambda^{<\kappa(\omega)}$-many $\mathbb{V}_\alpha$-colored trees up to bi-embeddability. Note that for all $\mathcal{T}, \mathcal{T}' \in A$ with $\mathcal{T}\not \leq^{\mbox{\small{ct}}} \mathcal{T}'$, we have that $\mathcal{T} \not \leq_{\alpha_*}^{\mbox{ct}} \mathcal{T}'$, since $A$ is $f$-closed. Hence $\mathcal{T} \times \alpha_* \not \leq^{\mbox{\small{ct}}} \mathcal{T}' \times \alpha_*$, by Lemma~\ref{RedToWF}. Hence we conclude by Lemma~\ref{TreeCountLemma}.
\end{proof}

This allows us to prove the following.
\begin{theorem}\label{TreeCount}
	Suppose $\kappa(\omega)$ exists, and $\alpha < \kappa(\omega)$, and $(\Phi, \sim_{\alpha \omega})$ has the Schr\"{o}der--Bernstein property. Then for all $\lambda$, $\tau(\Phi, \lambda) \leq \lambda^{<\kappa(\omega)}$. In particular, $\tau(\Phi, \kappa(\omega)) \leq \kappa(\omega)$. 
\end{theorem}
\begin{proof}
	Let $f: \mbox{CSS}(\Phi) \times \omega_1 \to \mbox{CT}_{\aleph_0}(\HC)$ be as in Lemma~\ref{EmbeddingAbsolute}. Let $g \in \mathbb{F}$ be as in Lemma~\ref{TreeCount0}. Let $h \in \mathbb{F}$ be defined by $h(a):=$ the foundation rank of $a$. We claim that $(f, g, h)$ witness that $\tau(\Phi, \lambda) \leq \lambda^{<\kappa(\omega)}$. That is, we show that if $A \in \mathbb{V}_{\lambda^+}$ is $f \times g \times h$-closed then $|\CSS(\Phi)_{\ptl} \cap A| \leq \lambda^{<\kappa(\omega)}$.

    Suppose towards a contradiction this failed. Then $|A| \geq \kappa(\omega)$, hence must contain an element $a \in A$ of foundation rank at least $\alpha$; since $A$ is $h$-closed, we get that $\alpha \in A$. Thus for all $\phi \in \CSS(\Phi)$, $f(\phi, \alpha) \in A$; by choice of $f$, these are pairwise non-bi-embeddable $\mathbb{V}_{\omega+\alpha+1}$-colored trees, contradicting the choice of $g$.
\end{proof}

\begin{corollary}\label{Quote1}
	Assume $\kappa(\omega)$ exists. Suppose $\Phi$ has the Schr\"{o}der--Bernstein property with respect to elementary embeddings, i.e. any two elementarily bi-embeddable countable models are isomorphic. Then $\Phi$ is not Borel complete. In fact, $\mbox{TAG}_1 \not \leq_B \Phi$.
\end{corollary}
\begin{proof}
The in fact clause follows from the fact that $\tau(\mbox{TAG}_1, \kappa(\omega)) = \beth_1(\kappa(\omega))$.
\end{proof}

\section{An Independence Result}\label{ExamplesSec}

\begin{definition}Suppose $\Phi \in \mathcal{L}_{\omega_1 \omega}$. Then say that $\Phi$ {\em admits Borel Schr\"{o}der--Bernstein invariants} if there is some $\Psi \in \mathcal{L}_{\omega_1 \omega}$ and some Borel reduction $f$ from $\Phi$ to $\Psi$, such that non-isomorphic pairs of models of $\Phi$ are sent to pairs of models of $\Psi$ which are not elementarily bi-embeddable. Formally, for all $M, N \in \mbox{Mod}(\Phi)$, if $M \cong N$ then $f(M) \cong f(N)$ and if $M \not \cong N$ then $f(M) \not \sim_{\omega \omega} f(N)$ (the same definition works for other choice of embeddings).

Similarly, say that $\Phi$ {\em admits $a\Delta^1_2$ Schr\"{o}der--Bernstien invariants} if there is some $\Psi \in \mathcal{L}_{\omega_1 \omega}$ and some absolutely $\Delta^1_2$ reduction from $\Phi$ to $\Psi$, which takes non-isomorphic models of $\Phi$ to non-bi-embeddable models of $\Psi$, and moreover, this continues to hold in every forcing extension.
\end{definition}

By Shoenfield's Absoluteness theorem, if $\Phi$ admits Borel Schr\"{o}der--Bernstein invariants then it admits $a \Delta^1_2$ Schr\"{o}der--Bernstein invariants.


\begin{theorem}
Suppose $\kappa(\omega)$ exists and $\Phi \in \mathcal{L}_{\omega_1 \omega}$. If $\Phi$ admits $a \Delta^1_2$-Schr\"{o}der--Bernstein invariants then $\tau(\Phi, \kappa(\omega)) \leq \kappa(\omega)$, hence it is not Borel complete.
\end{theorem}
\begin{proof}
    Let $g: \mbox{Mod}(\Phi) \to \mbox{Mod}(\Psi)$ witness that $\Phi$ admits $a \Delta^1_2$-Schr\"{o}der--Bernstein invariants.  Let $g': \CSS(\Phi) \to \CSS(\Psi)$ be the corresponding persistent injection. Let $f$ be as in Lemma \ref{EmbeddingAbsolute}, and let $f': \CSS(\Psi) \to \mbox{CT}_{\aleph_0}(\HC_{\omega + \omega+1})$ be given by $f'(\psi) = f(\psi, 0)$. Then $f' \circ g'$ is an injection from $\CSS(\Phi)$ to $\mbox{CT}_{\aleph_0}(\HC_{\omega + \omega+1})$ that takes distinct $\phi, \phi' \in \CSS(\Phi)$ to non-bi-embeddable colored trees. 

    Put $f_* = (f' \circ g') \times h$ where $h$ is as in Lemma~\ref{TreeCount0}. We claim $f_*$ witnesses $\tau(\Phi, \kappa(\omega)) \leq \kappa(\omega)$. Suppose $A \in \mathbb{V}_{\lambda^+}$ is $f_*$-closed. Then $(f' \circ g')_{\ptl}$ gives an injection from $\CSS(\Phi)_{\ptl} \cap A$ to $\mathbb{V}_{\omega + \omega + 1}$-colored trees taking inequal sentences to non-bi-embeddable colored trees, thus $|\CSS(\Phi)_{\ptl}| \leq \kappa(\omega)$ by choice of $h$.
\end{proof}

\begin{corollary}
    Suppose ``$ZFC + \kappa(\omega)$ exists" is consistent. Then it is independent of $ZFC$ whether Graphs admits $a \Delta^1_2$-Schr\"{o}der--Bernstein invariants. 
\end{corollary}
\begin{proof}
    In \cite{TFAGWShelah}, Theorem 26, we show that it is consistent with $ZFC$ that Graphs admits $a \Delta^1_2$-Schr\"{o}der--Bernstein invariants, specifically it follows from the non-existence of any transitive model of $ZFC^-$ + $\kappa(\omega)$-exists. More precisely, Theorem 26 says that under this hypothesis, there is an $a \Delta^1_2$ reduction from Graphs to $\omega$-colored trees, which takes nonisomorphic graphs to non-$\leq^{\mbox{\small{ct}}}$-bi-embeddable (and hence non-elementarily-bi-embeddable) trees.

    On the other hand, by the preceding theorem, if $\kappa(\omega)$ exists, then Graphs cannot admit Schr\"{o}der--Bernstein invariants, since $\tau(\mbox{Graphs}, \kappa(\omega)) = \beth_{\kappa(\omega)^+}$. Thus we get independence.
\end{proof}

We thus see it is consistent that Graphs do not admit Borel Schr\"{o}der--Bernstein invariants; it is open whether this can be proven in $ZFC$.

\end{document}